\documentclass[11pt,reqno]{amsart}
\usepackage{amsmath,amsxtra,latexsym,amsthm,amssymb,amscd,pb-diagram}
\usepackage{mathrsfs,mathabx,amsfonts}
\usepackage[colorlinks=true,linkcolor=magenta,citecolor=blue]{hyperref}
\usepackage[margin=1in]{geometry}

\usepackage{bm}
\usepackage{color}
\usepackage{makecell,multirow,diagbox}
\usepackage{mathtools}
\usepackage[displaymath, mathlines]{lineno}
\usepackage{tikz}
\usepackage{graphicx}
\usepackage{epsfig}
\usepackage{array}
\usepackage{enumitem}
\usepackage{float}

\newtheorem{definition}{Definition}[section]
\newtheorem{theorem}{Theorem}[section]
\newtheorem{proposition}{Proposition}[section]
\newtheorem{lemma}{Lemma}[section]

\newtheorem{remark}{Remark}[section]

\newcommand{\R}{\mathbb{R}}

\newcommand{\Z}{\mathbb{Z}}

\newcommand{\e}{\varepsilon}
\newcommand{\ity}{\infty}

\begin{document}
\title[Critical curve for the system with different damping mechanisms]{Critical curve for a weakly coupled system of semi-linear $\sigma$-evolution equations with different damping types}

\thanks{$^* $\textit{Corresponding author:} Tuan Anh Dao (anh.daotuan@hust.edu.vn)}

\maketitle
\centerline{\scshape Dinh Van Duong$^1$, Tuan Anh Dao$^{1,*}$, Michael Reissig$^2$}
\medskip
{\footnotesize
	\centerline{$^1$ Faculty of Mathematics and Informatics, Hanoi University of Science and Technology}
	\centerline{No.1 Dai Co Viet road, Hanoi, Vietnam}
	\centerline{$^2$ Faculty for Mathematics and Computer Science, TU Bergakademie Freiberg}
	\centerline{Pr\"{u}ferstr. 9, 09596, Freiberg, Germany}}

\begin{abstract}
In this paper, we would like to consider the Cauchy problem for a weakly coupled system of semi-linear $\sigma$-evolution equations with different damping mechanisms for any $\sigma > 1$, ``parabolic like damping'' and ``$\sigma$-evolution like damping''. Motivated strongly by the well-known Nakao's problem, the main goal of this work is to determine the critical curve between the power exponents $p$ and $q$ of nonlinear terms by not only establishing the global well-posedness property of small data solutions but also indicating blow-up in finite time weak solutions. We want to point out that the application of a modified test function associated with a judicious choice of test functions really plays an essential role to show a blow-up result for solutions and upper bound estimates for lifespan of solutions, where $\sigma$ is assumed to be any fractional real number. To end this paper, lower bound estimates for lifespan of Sobolev solutions are also shown to verify their sharp results in some spatial dimensions.
\end{abstract}

\tableofcontents

\section{Introduction}
Let us consider the following Cauchy problem for the following weakly coupled system of semi-linear $\sigma$-evolution equations with mixing two different damping types:
\begin{equation} \label{Main.Eq.1}
\begin{cases}
u_{tt} + (-\Delta)^\sigma u+ (-\Delta)^{\sigma} u_t = |v|^p, & x\in \R^n,\, t> 0, \\
v_{tt} + (-\Delta)^{\sigma} v + v_t = |u|^q, & x \in \mathbb{R}^n,\, t> 0,\\
u(0,x)= 0,\quad u_t(0,x)= u_1(x), & x\in \R^n, \\
v(0,x) = 0,\, \quad v_t(0,x) = v_1(x), & x \in \mathbb{R}^n,
\end{cases}
\end{equation}
where $\sigma > 1$ is assumed to be any real number. The parameters $p,q>1$ stand for power exponents of the nonlinear terms. The operator $(-\Delta)^{\sigma}$ is defined by
\begin{align*}
    \left((-\Delta)^{\sigma} \varphi\right)(x) := \mathfrak{F}^{-1}_{\xi \to x}\left(|\xi|^{2\sigma}\widehat{\varphi}(\xi)\right)(x),
\end{align*}
where $\widehat{\varphi}(\xi):= \mathfrak{F}_{x\rightarrow \xi}\big(\varphi(x)\big)$ stands for the Fourier transform with respect to the space variables of a function $\varphi=\varphi(x)$. Moreover, $(-\Delta)^{\sigma} u_t$ and $v_t$ together appearing in \eqref{Main.Eq.1} represent two damping terms, called visco-elastic type (or strong) damping and frictional (or external) damping, respectively. \medskip

During recent years, the following linear Cauchy problem for $\sigma$-evolution equations has gained a lot of attention from many mathematicians over the world (see, for example, \cite{Matsumura77,Shibata2000,DabbiccoReissig2014,DuongKainaneReissig2015,DabbiccoEbert2017,DaoReissig1,DaoReissig2}):
\begin{align}\label{Eq2}
    \begin{cases}
        u_{tt} +(-\Delta)^{\sigma} u + \mu(-\Delta)^{\delta} u_t =0, & x \in \mathbb{R}^n,\, t > 0,\\
        u(0,x) = u_0(x), \quad u_t(0,x) = u_1(x), & x \in \mathbb{R}^n,
    \end{cases}
\end{align}
with $\sigma \geq 1, \delta \in [0,\sigma]$ and $\mu > 0$. This family of models is quite interesting by means of choosing particular parameters $\sigma$ and $\delta$ in \eqref{Eq2}.
Among other things, one recognizes that properties of solutions to \eqref{Eq2} change completely with respect to decay estimates when $\delta$ is taken from $[0, \sigma/2)$ to $(\sigma/2, \sigma]$. From this observation, the authors in the cited papers proposed a classification between ``parabolic like models" with $\delta \in [0, \sigma/2)$, and ``$\sigma$-evolution like models" with $\delta \in (\sigma/2, \sigma]$, the so-called ``wave like models'' in the special case $\sigma = 1$.
Moreover, according to the asymptotic profile aspects of solutions to \eqref{Eq2}, the solution behavior, as $t \to \ity$, of ``parabolic like models" is similar to that of the diffusion equation
\begin{equation*}
\begin{cases}
    v_t+ \nu (-\Delta)^{\sigma- \delta}v= 0, & x\in \R^n,\, t> 0, \\
    v(0,x)= v_0(x), & x\in \R^n,
\end{cases}
\end{equation*}
where the initial data $v_0=v_0(u_0,u_1,\mu,\sigma,\delta,n)$ and a positive constant $\nu=\nu(\mu,\sigma,\delta,n)$ are appropriately chosen. Clearly, the above diffusion equation becomes the heat equation when the situation $\sigma-\delta=1$ occurs. Meanwhile, this phenomenon is no longer true for ``$\sigma$-evolution like models", i.e. some kinds of wave structure and oscillations in time appear to describe the asymptotic profile of solutions to \eqref{Eq2}. More precisely, the solution behavior, as $t \to \ity$, of ``$\sigma$-evolution like models" is related to that of the free evolution equation
\begin{equation*}
\begin{cases}
    v_{tt}+ \nu (-\Delta)^\sigma v= 0, & x\in \R^n,\, t> 0, \\
    v(0,x)= v_0(x),\quad v_t(0,x)= v_1(x), & x\in \R^n,
\end{cases}
\end{equation*}
with suitable initial data and for some constant $\nu>0$. From these above-mentioned views, we may claim that the specific case $\delta=\sigma/2$ is understood as a threshold to distinguish \eqref{Eq2} into two different families of models corresponding to  $\delta \in [0,\sigma/2)$ and $\delta \in (\sigma/2,\sigma]$.
\medskip

Next, to present our motivation in terms of studying the Cauchy problem \eqref{Main.Eq.1} let us recall some previous papers (see \cite{Wakasugi2017,ChenReissig2021,Chen2022,KitaKusaba2022,PalmieriTakamura2023}), by taking account of the following weakly coupled system for a semi-linear damped wave equation and a semi-linear wave equation in the whole space:
\begin{equation}\label{Eq3}
\begin{cases}
   u_{tt} - \Delta u + u_t = |v|^p, & x \in \mathbb{R}^n,\, t > 0,\\
   v_{tt} - \Delta v = |u|^q, & x \in \mathbb{R}^n,\, t > 0,\\
   u(0,x) = u_0(x), \quad u_t(0,x) = u_1(x), & x \in \mathbb{R}^n,\\
   v(0,x) = v_0(x), \quad v_t(0,x) = v_1(x), & x \in \mathbb{R}^n.
   \end{cases}
\end{equation}
To determine the \textit{critical curve} between the exponents $p$ and $q$ of the power nonlinearities is a part of the so-called Nakao's problem proposed by Professor Mitsuhiro Nakao (Emeritus of Kyushu University). Here, the critical curve expressed by $\Gamma(n,p,q)=0$ in the $p-q$ plane is understood as the threshold condition of a pair of exponents $(p,q)$ between global (in time) existence of small data solutions (stability of the zero solutions) if $\Gamma(n,p,q)<0$, and blow-up in finite time solutions even for small data if $\Gamma(n,p,q)>0$. More precisely, the author in \cite{Wakasugi2017} used the test function method to find out
$$ \Gamma(n,p,q)= \max\left\{\frac{q/2+1}{pq-1}-\frac{n-1}{2}, \frac{q+1}{pq-1}-\frac{n}{2}, \frac{p+1}{pq-1}-\frac{n}{2}\right\}, $$
which becomes sharp in one-dimensional case only, however, seems to be not optimal in higher dimensions $n\ge 2$. For this reason, a remarkable improvement was established in the paper \cite{ChenReissig2021} by effectively applying an iteration argument to partially fill this lack, i.e.
$$ \Gamma(n,p,q)= \max\left\{\frac{q/2+1}{pq-1}-\frac{n-1}{2}, \frac{p^{-1}+2}{pq-1}-\frac{n-1}{2}, \frac{p+1/2}{pq-1}-\frac{n}{2}\right\}. $$
Then, a further contribution coming from \cite{KitaKusaba2022} is to extend the above-mentioned results in \cite{Wakasugi2017,ChenReissig2021}, particularly, the authors in \cite{KitaKusaba2022} found out a combined result as the union of blow-up ranges obtaining from \cite{Wakasugi2017,ChenReissig2021}. Quite recently, the authors in \cite{Chen2022,PalmieriTakamura2023} considered a Nakao-type weakly coupled system with nonlinearities of derivative type instead of \eqref{Eq3}. Namely, some results for blow-up and lifespan in the subcritical case were demonstrated in \cite{Chen2022} by an iteration argument, afterwards, the authors in \cite{PalmieriTakamura2023} explored such results in the critical case based on the Zhou's method (see more \cite{{Zhou2001}}). Pay attention that in \cite{PalmieriTakamura2023} they concerned \eqref{Eq3} with a semi-linear damped Klein-Gordon equation in place of a semi-linear damped equation on the left-hand side and a pair of nonlinearities $\{|v_t|^p,|u_t|^q\}$ on the right-hand sides. Their interest is to indicate a blow-up result in finite time under suitable sign assumptions for the Cauchy data when the exponents $(p,q)$ belong to a suitable range. Among other things, one recognizes how the different structure of damped wave equations (parabolic like) and wave equations (hyperbolic like) affects on the critical curve for a weakly coupled system of their corresponding semi-linear equations. \medskip

Inspired strongly by the cited papers, in this work we would like to investigate the influence of two kind of different equations, ``parabolic like models" and ``$\sigma$-evolution like models", on their weakly coupled system \eqref{Main.Eq.1} with the usual power nonlinearities. The main purpose of this paper is to verify the following critical curve for \eqref{Main.Eq.1}:
$$ \max\left\{\frac{2q+1}{pq +q-2}, \frac{pq+p+1}{2pq-p-1} \right\} - \frac{n}{2\sigma}=0 $$
by providing both global (in time) existence of small data Sobolev solutions (Theorem \ref{Theorem1}) and blow-up in finite time result (Theorem \ref{Theorem3}). We want to stress out that the appearance of damping terms in \eqref{Main.Eq.1} gives some benefits in proving a result for global (in time) existence of Sobolev solutions, which never appeared in the cited papers to explore \eqref{Eq3}. At this point, the crux of our approach is based on the technique of using loss of decay associated with recently developed tools from Harmonic Analysis. The advantage worthy of mentioning of allowing some loss of decay is to show how the restrictions to the admissible exponents $p$ and $q$ could be relaxed. Additionally, for the second contribution of this work we are interested in reporting sharp estimates for lifespan of Sobolev solutions by demonstrating their lower bound and upper bound simultaneously when a blow-up situation of solutions occurs.

\textbf{Notations}
\begin{itemize}[leftmargin=*]
\item We write $f\lesssim g$ when there exists a constant $C>0$ such that $f\le Cg$, and $f \approx g$ when $g\lesssim f\lesssim g$.
\item  The spaces $H^a$ and $\dot{H}^a$, with $a \ge 0$, stand for Bessel and Riesz potential spaces based on $L^2$ spaces. Here $\big<D\big>^{a}$ and $|D|^{a}$ denote the differential operator with symbol $\big<\xi\big>^{a}$ and the fractional Laplace operator with symbol $|\xi|^{a}$, respectively.
\item For a given number $s \in \R$, we denote by
$$ [s]:= \max \big\{k \in \Z \,\, : \,\, k\le s \big\} \quad \text{ and }\quad [s]^+:= \max\{s,0\}, $$
its integer part and its positive part, respectively.
\item We put $\big< x\big>:= \sqrt{1+|x|^2}$, the so-called Japanese bracket of $x \in \R^n$.
\item Finally, we introduce the space
$\mathcal{D}:= (L^1 \cap L^2) \times (L^1 \cap L^2)$ with the norm
$$\|(f,g)\|_{\mathcal{D}}:= \|f\|_{L^1} + \|f\|_{L^2} + \|g\|_{L^1}+ \|g\|_{L^2}. $$

\end{itemize}

\textbf{Main results} \medskip

Let us state the global (in time) existence of small data solutions and the blow-up result, which will be proved in this paper.
\begin{theorem}[\textbf{Global existence result}]\label{Theorem1}
    Let $ 1 < \sigma < n < 2\sigma$. We assume that the exponents $p, q$ satisfy the following conditions:
    \begin{align}\label{condition1.1.1}
        \max\left\{\frac{2q+1}{pq+q-2},\frac{pq+p+1}{2pq-p-1}\right\}< \frac{n}{2\sigma}  \text{ and } p \geq 2.
    \end{align}
Then, there exists a constant $\varepsilon_0 > 0$ such that for any small data
$ (u_1,v_1) \in \mathcal{D} $
fulfilling the assumption $ \|(u_1,v_1)\|_{\mathcal{D}} < \varepsilon_0$,
we have a uniquely determined global (in time) small data Sobolev solution
\begin{align*}
    (u,v) \in \left(\mathcal{C}([0,\infty), H^{\sigma}\cap L^{\infty} \cap L^q) \cap \mathcal{C}^1([0,\infty), L^2)\right) \times \left(\mathcal{C}([0,\infty), H^{\sigma} ) \cap \mathcal{C}^1([0,\infty), L^2)\right)
\end{align*}
to \eqref{Main.Eq.1}. Moreover, the following estimates hold:
\begin{align}
\|u(t,\cdot)\|_{L^q} &\lesssim (1+t)^{1-\frac{n}{\sigma}(1-\frac{1}{q})+[\varepsilon_1(p)]^+} \|(u_1,v_1)\|_{\mathcal{D}}, \notag\\
\|u(t,\cdot)\|_{L^{\infty}} &\lesssim (1+t)^{1-\frac{n}{\sigma}+[\varepsilon_1(p)]^+} \|(u_1,v_1)\|_{\mathcal{D}}, \notag\\
\||D|^{\sigma}u(t,\cdot)\|_{L^2} &\lesssim (1+t)^{-\frac{n}{4\sigma}+[\varepsilon_1(p)]^+} \|(u_1,v_1)\|_{\mathcal{D}},\notag\\
\|u_t(t,\cdot)\|_{L^2} &\lesssim (1+t)^{-\frac{n}{4\sigma}+[\varepsilon_1(p)]^+} \|(u_1,v_1)\|_{\mathcal{D}},\notag\\
\|v(t,\cdot)\|_{L^2} &\lesssim (1+t)^{-\frac{n}{4\sigma}+[\varepsilon_2(q)]^+} \|(u_1,v_1)\|_{\mathcal{D}},\notag\\
\||D|^{\sigma}v(t,\cdot)\|_{L^2} &\lesssim (1+t)^{-\frac{n}{4\sigma}-\frac{1}{2}+[\varepsilon_2(q)]^+} \|(u_1,v_1)\|_{\mathcal{D}},\notag\\
\|v_t(t,\cdot)\|_{L^2} &\lesssim (1+t)^{-\frac{n}{2\sigma}+[\varepsilon_2(q)]^+} \|(u_1,v_1)\|_{\mathcal{D}}, \notag
\end{align}
where
$\varepsilon_1(p) :=
       1 -\frac{n}{2\sigma}(p-1)+\varepsilon$
 and
 $\varepsilon_2(q) := 1+q-\frac{n}{\sigma}(q-1)+\varepsilon$ for any small positive number $\varepsilon$. 
\end{theorem}

\begin{remark}
\fontshape{n}
\selectfont
It is obvious to see that one of two quantities $[\e_1(p)]^+$ and $[\e_2(q)]^+$ appearing in Theorem \ref{Theorem1} is non-negative due to the condition \eqref{condition1.1.1}. Comparing the achieved estimates for solutions to (\ref{Main.Eq.1}) with the corresponding ones to the linear equations \eqref{Problem1} and \eqref{Problem2} (see more Propositions \ref{Proposition2.1} and \ref{Proposition2.2}), we can understand that the estimates in Theorem \ref{Theorem1}, especially $\|v_t(t,\cdot)\|_{L^2}$, exhibit a certain loss of decay.
\end{remark}

\begin{remark}
\fontshape{n}
\selectfont
    The restrictions on the space dimension and the appearance of the condition $p \geq 2$ and $n < 2\sigma$ in Theorem \ref{Theorem1} are due to the application of the Gagliardo-Nirenberg inequality. Moreover, the condition $q \geq 2$ also arises when using this tool. However, it disappears because of the range of $(p,q)$ considered under condition (\ref{condition1.1.1}).
\end{remark}

\begin{theorem}[\textbf{Blow-up result}]\label{Theorem3}
    Let $\sigma\ge 1$ and $p,q > 1$. Assume that $u_1, v_1 \in L^1$ enjoy the conditions
    \begin{align}\label{condition1.3.1}
       \int_{\mathbb{R}^n} u_1(x) dx > 0 \text{ and } \int_{\mathbb{R}^n} v_1(x) dx > 0 .
    \end{align}
    Furthermore, we suppose the assumptions $n \leq \sigma$ or
    \begin{align}
     \displaystyle \max\left\{\frac{2q+1}{pq +q-2}, \frac{pq+p+1}{2pq-p-1} \right\} > \frac{n}{2\sigma} \text{ if } n > \sigma.\label{condition1.3.3}
     \end{align}
    Then, there is no global (in time) weak solution $(u,v) \in (\mathcal{C}([0,\infty), L^q) \cap L^q([0, \infty) \times \mathbb{R}^n)) \times \mathcal{C}([0,\infty), L^2)$ to \eqref{Main.Eq.1}. 
\end{theorem}

\begin{remark}
	\fontshape{n}
	\selectfont
From the conditions \eqref{condition1.1.1} and \eqref{condition1.3.3} in Theorems \ref{Theorem1} and \ref{Theorem3}, respectively, we claim that the critical curve for \eqref{Main.Eq.1} in the $p-q$ plane is precisely described by
$$ \max\left\{\frac{2q+1}{pq +q-2}, \frac{pq+p+1}{2pq-p-1} \right\} -\frac{n}{2\sigma}= 0. $$
\end{remark}

\begin{remark}
	\fontshape{n}
	\selectfont
For the purpose of observing more explicitly, let us illustrate some ranges describing results for both global existence from Theorem \ref{Theorem1} and blow-up from Theorem \ref{Theorem3} in the $p-q$ plane in the following figure:
\begin{figure}[H]
\begin{center}
\begin{tikzpicture}[>=latex,xscale=1.0,scale=0.9]

\draw[->] (0,0) -- (7.5,0)node[below]{$p$};
\draw[->] (0,0) -- (0,7.5)node[left]{$q$};
\node[below left] at(0,0){$0$};
\node[below] at (2.5,1.0) {};


\draw[fill] (1,0) circle[radius=1pt];
\node[below] at (1,0){{\scriptsize $1$}};
\draw[dashed] (1,0)--(1,5);

\draw[fill] (2,0) circle[radius=1pt];
\node[below] at (2,0){{\scriptsize $2$}};
\draw[dashed] (2,0)--(2,5);

\draw[fill] (2.7,0) circle[radius=1pt];
\node[below] at (2.7,0){{\scriptsize $p_{\rm crit}$}};
\draw[dashed] (2.7,0)--(2.7,5);

\draw[fill] (0,1) circle[radius=1pt];
\node[left] at (0,1){{\scriptsize $1$}};
\draw[dashed] (0,1)--(7,1);

\draw[fill] (0,2) circle[radius=1pt];
\node[left] at (0,2){{\scriptsize $2$}};
\draw[fill] (0,3) circle[radius=1pt];
\node[left] at (0,3){{\scriptsize $q_{\rm crit}$}};
\draw[fill] (0,2.3) circle[radius=1pt];
\node[left] at (0,2.3){{\scriptsize $q_0$}};
\draw[fill] (1.3,0) circle[radius=1pt];
\node[below] at (1.3,0){{\scriptsize $p_0$}};

\fill[color=black!10!white] (1,7)--(1,1)--(2.7,1)--(2.7,3)--(2.2,3.95628)--(2,4.7213)--(1.72,7.01637)--cycle;

\fill[color=black!10!white] (2.7,3)--(2.7,1)--(7,1)--(7,2.5)--(4,2.734)--cycle;

\fill[color=black!25!white] (2,7)--(2,4.7213)--(2.2,3.95628)--(2.5,3.28688)--(2.7,3)--(7,3)--(7,7)--cycle;

\fill[color=black!25!white] (2.7,3)--(4,2.734)--(7,2.5)--(7,3)--cycle;

\draw[domain = 1.72:2.7,blue,line width=1.0pt]plot(\x,{(1.2787*(\x)+0.74751)/((\x)-1.3)});
\draw[domain = 2.7:6, dashed, blue,line width=1.0pt]plot(\x,{(1.2787*(\x)+0.74751)/((\x)-1.3)});

\draw[domain = 2.7:7,red,line width=1.0pt]plot(\x,{(0.15+2.3*(\x))/((\x)-0.58)});
\draw[domain = 1.5:2.7,red, dashed, line width=1.0pt]plot(\x,{(0.15+2.3*(\x))/((\x)-0.58)});

\draw[fill] (2.7,3) circle[radius=1.5pt];
\draw[dashed] (1,0)--(1,7);
\draw[dashed] (0,3)--(7,3);
\draw[dashed] (0,1)--(7,1);
\draw[thin] (2,4.7213)--(2,7);

\draw[thin] (1,1)--(1,7);
\draw[thin] (1,1)--(7,1);
\draw[dashed] (2,0)--(2,7);

\draw[dashed] (0,2)--(7,2);
\draw[dashed] (2.7,0)--(2.7,7);
\draw[dashed] (0,3)--(2.7,3);

\draw[thin] (0,0)--(0,7.3);
\draw[thin] (0,0)--(7.3,0);
\draw[dashed] (0,2.3)--(7,2.3);
\draw[dashed] (1.3,0)--(1.3,7);

\fill[color=black!10!white] (9.25,3.75)--(8.75,3.75)--(8.75,3.25)--(9.25,3.25)--cycle;
\node[right] at (9.4,3.5) {{\footnotesize \text{Blow-up}}};

\fill[color=black!25!white] (9.25,2.75)--(8.75,2.75)--(8.75,2.25)--(9.25,2.25)--cycle;
\node[right] at (9.4,2.5) {{\footnotesize \text{Global existence}}};

\draw[thin, color=blue,line width=1.0pt] (9.25,4.75)--(8.75,4.75);
\node[right] at (9.5,4.75) {$\dfrac{2q+1}{pq +q-2} = \dfrac{n}{2\sigma}$};

\draw[thin, color=red,line width=1.0pt] (9.25,5.75)--(8.75,5.75);
\node[right] at (9.5,5.75) {$\dfrac{pq+p+1}{2pq-p-1} = \dfrac{n}{2\sigma}$};
\end{tikzpicture}
\caption{Global existence and blow-up results in the $p-q$ plane when $\frac{4\sigma}{3} < n < 2\sigma$.}
\label{fig.zone1}
\end{center}
\end{figure}
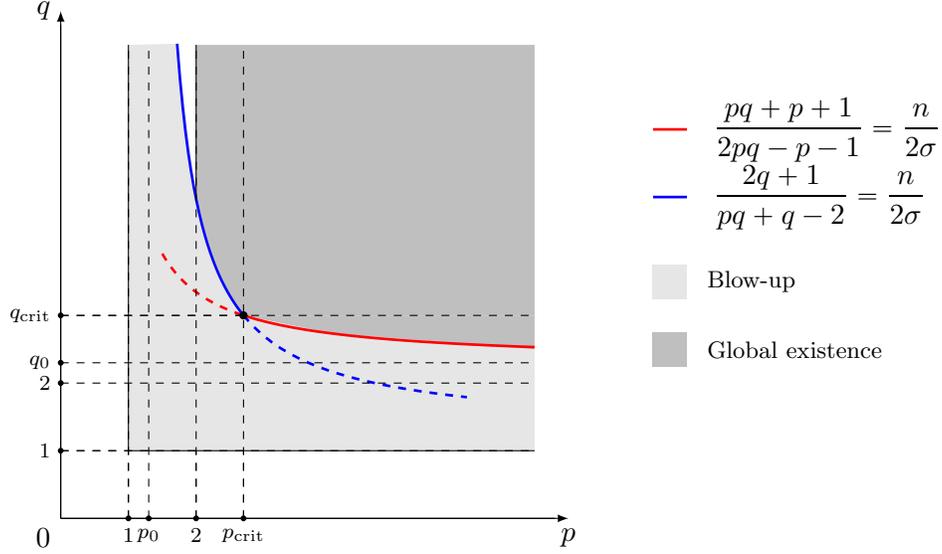
\end{remark}

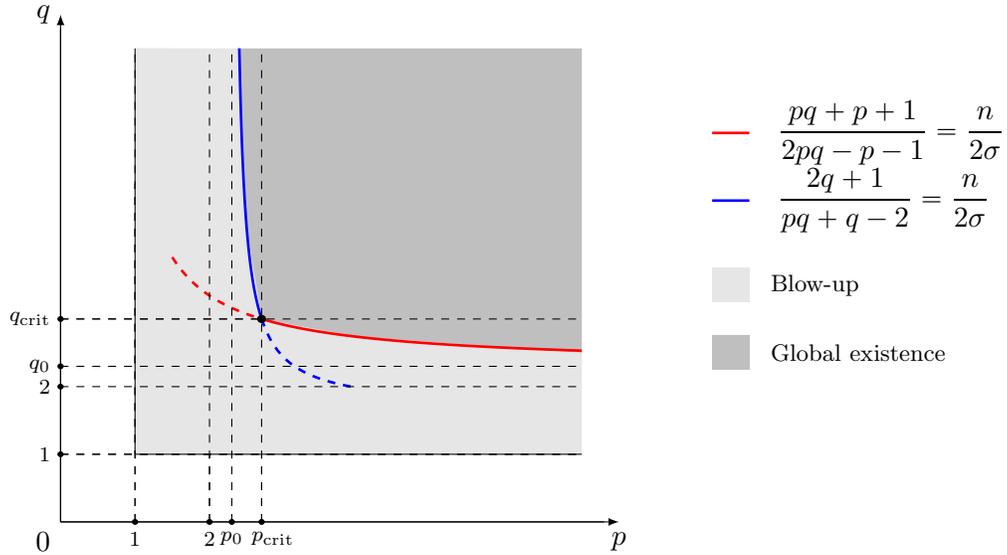
\begin{figure}[H]
\begin{center}
\begin{tikzpicture}[>=latex,xscale=1.1,scale=0.9]

\draw[->] (0,0) -- (7.5,0)node[below]{$p$};
\draw[->] (0,0) -- (0,7.5)node[left]{$q$};
\node[below left] at(0,0){$0$};
\node[below] at (2.5,1.0) {};


\draw[fill] (1,0) circle[radius=1pt];
\node[below] at (1,0){{\scriptsize $1$}};
\draw[dashed] (1,0)--(1,5);

\draw[fill] (2,0) circle[radius=1pt];
\node[below] at (2,0){{\scriptsize $2$}};
\draw[dashed] (2,0)--(2,5);

\draw[fill] (2.7,0) circle[radius=1pt];
\node[below] at (2.85,0){{\scriptsize $p_{\rm crit}$}};
\draw[dashed] (2.7,0)--(2.7,5);

\draw[fill] (0,1) circle[radius=1pt];
\node[left] at (0,1){{\scriptsize $1$}};
\draw[dashed] (0,1)--(7,1);

\draw[fill] (0,2) circle[radius=1pt];
\node[left] at (0,2){{\scriptsize $2$}};
\draw[fill] (0,3) circle[radius=1pt];
\node[left] at (0,3){{\scriptsize $q_{\rm crit}$}};
\draw[fill] (0,2.3) circle[radius=1pt];
\node[left] at (0,2.3){{\scriptsize $q_0$}};
\draw[fill] (2.3,0) circle[radius=1pt];
\node[below] at (2.3,0){{\scriptsize $p_0$}};

\fill[color=black!10!white] (1,7)--(1,1)--(2.7,1)--(2.7,3)--(2,4.7213)--(1.5,7)--cycle;

\fill[color=black!10!white] (1,7)--(2.4,7)--(2.5,4.33333)--(2.7,3)--(1,3)--cycle;

\fill[color=black!10!white] (2.7,3)--(2.7,1)--(7,1)--(7,2.5)--(4,2.734)--cycle;

\fill[color=black!25!white] (2.7,3)--(4,2.734)--(7,2.5)--(7,3)--cycle;

\fill[color=black!25!white] (1,7)--(2.4,7)--(2.5,4.33333)--(2.7,3)--(7,3)--(7,7)--cycle;

\draw[domain = 2.4:2.7, blue,line width=1.0pt]plot(\x,{(1.666667*(\x)-3.3)/((\x)-2.3)});

\draw[domain = 2.7:4, dashed, blue,line width=1.0pt]plot(\x,{(1.666667*(\x)-3.3)/((\x)-2.3)});

\draw[domain = 2.7:7,red,line width=1.0pt]plot(\x,{(0.15+2.3*(\x))/((\x)-0.58)});
\draw[domain = 1.5:2.7,red, dashed, line width=1.0pt]plot(\x,{(0.15+2.3*(\x))/((\x)-0.58)});

\draw[fill] (2.7,3) circle[radius=1.5pt];
\draw[dashed] (1,0)--(1,7);
\draw[dashed] (0,3)--(7,3);
\draw[dashed] (0,1)--(7,1);

\draw[thin] (1,1)--(1,7);
\draw[thin] (1,1)--(7,1);
\draw[dashed] (2,0)--(2,7);

\draw[dashed] (0,2)--(7,2);
\draw[dashed] (2.7,0)--(2.7,7);
\draw[dashed] (0,3)--(2.7,3);

\draw[thin] (0,0)--(0,7.3);
\draw[thin] (0,0)--(7.3,0);
\draw[dashed] (0,2.3)--(7,2.3);
\draw[dashed] (2.3,0)--(2.3,7);

\fill[color=black!10!white] (9.25,3.75)--(8.75,3.75)--(8.75,3.25)--(9.25,3.25)--cycle;
\node[right] at (9.4,3.5) {{\footnotesize \text{Blow-up}}};

\fill[color=black!25!white] (9.25,2.75)--(8.75,2.75)--(8.75,2.25)--(9.25,2.25)--cycle;
\node[right] at (9.4,2.5) {{\footnotesize \text{Global existence}}};

\draw[thin, color=blue,line width=1.0pt] (9.25,4.75)--(8.75,4.75);
\node[right] at (9.5,4.75) {$\dfrac{2q+1}{pq +q-2} = \dfrac{n}{2\sigma}$};

\draw[thin, color=red,line width=1.0pt] (9.25,5.75)--(8.75,5.75);
\node[right] at (9.5,5.75) {$\dfrac{pq+p+1}{2pq-p-1} = \dfrac{n}{2\sigma}$};
\end{tikzpicture}
\caption{Global existence and blow-up results in the $p-q$ plane when $\sigma < n \leq \frac{4\sigma}{3}$.}
\label{fig.zone2}
\end{center}
\end{figure}
Here $p_{\rm crit}:= 1+\displaystyle\frac{2\sigma}{n}$, $q_{\rm crit}:= 1+\displaystyle\frac{2\sigma}{n-\sigma}$, moreover, $p = p_0 := -1+\displaystyle\frac{4\sigma}{n}$, $q = q_0 := \displaystyle\frac{n+2\sigma}{2(n-\sigma)}$ are asymptotic lines corresponding to the curves $$\displaystyle\frac{2q+1}{pq+q-2}=\frac{n}{2\sigma},\quad \text{ i.e. }\quad q = \displaystyle\frac{2(n+\sigma)}{n(p+1)-4\sigma}$$ and $$\displaystyle\frac{pq+p+1}{2pq-p-1}=\frac{n}{2\sigma},\quad \text{ i.e. }\quad q = \displaystyle\frac{(p+1)(n+2\sigma)}{2p(n-\sigma)},$$ respectively. We note that
\begin{align}\label{relation10}
   \displaystyle\frac{(p+1)(n+2\sigma)}{2p(n-\sigma)} \geq \displaystyle\frac{2(n+\sigma)}{n(p+1)-4\sigma} \quad\text{ if and only if }\quad p \geq p_{\rm crit} = 1 + \frac{2\sigma}{n}.
\end{align}

\textbf{The remaining part of this paper is organized as follows:} In Section \ref{section3}, we are going to recall known results in previous studies on the Cauchy problems for the corresponding linear equations included in the system (\ref{Main.Eq.1}) to prove global (in time) existence of small data solutions to (\ref{Main.Eq.1}). Then, we will present the proof of the blow-up result for finite (in time) Sobolev solutions in Section \ref{section4}. Finally, Section \ref{Sect.Final} not only devotes to some estimates for lifespan of solutions but also leaves an open problem involving (\ref{Main.Eq.1}).

\section{Global (in time) existence of small data solutions}\label{section3}
\subsection{Auxiliary estimates for solution to the linear equations}\label{section2}
In this section, we will recall useful estimates in previous papers. For this purpose, let us turn to the following linear Cauchy problem:
\begin{equation}\label{Problem1}
    \begin{cases}
        u_{tt} + (-\Delta)^{\sigma} u_t + (-\Delta)^{\sigma} u = 0, &\text{ } x \in \mathbb{R}^n, t > 0,\\
        u(0,x) = 0, \quad u_t(0,x) = u_1(x), &\text{ } x \in \mathbb{R}^n,
    \end{cases}
\end{equation}
where $\sigma > 0$. The solution of the problem (\ref{Problem1}) is given by
\begin{align}
     u^{\rm lin}(t,x) :=  \mathfrak{F}^{-1}_{\xi \to x}\left(\widehat{\mathcal{K}_{1}}(t,\xi)  \widehat{u_1}(\xi)\right), \label{solution1}
\end{align}
where
\begin{align*}
    \widehat{\mathcal{K}_1}(t,\xi) = \mathfrak{F}_{x \to \xi}(\mathcal{K}_1(t,x)) =  \frac{e^{\lambda_{11}(\xi) t}-e^{\lambda_{12}(\xi) t}}{\lambda_{11}(\xi)-\lambda_{12}(\xi)},
\end{align*}
in which the characteristic roots are determined as follows:
\begin{align*}
    \lambda_{11, 12} :=
    \begin{cases}
        \displaystyle\frac{1}{2} \left(-|\xi|^{2\sigma} \pm |\xi|^{\sigma}\sqrt{|\xi|^{2\sigma}-4}\right) &\text{ if } |\xi| \geq 2^{\frac{1}{\sigma}}, \\
        \\
        \displaystyle\frac{1}{2} \left(-|\xi|^{2\sigma} \pm i|\xi|^{\sigma}\sqrt{4-|\xi|^{2\sigma}}\right) &\text{ if } |\xi| < 2^{\frac{1}{\sigma}}.
    \end{cases}
\end{align*}
\begin{proposition}[see Theorems 2.1 and 9.1 in \cite{DabbiccoEbert2021}] \label{Proposition2.1} Let $\sigma>0$ and $\sigma\ne 1$. Assuming that $m \in [1,2]$ and $u_1 \in L^m \cap L^2$. Then, the following energy estimate for solutions to \eqref{Problem1} holds:
\begin{align*}
    \|\big(|D|^{\sigma},\partial_t\big) u(t,\cdot)\|_{L^2} &\lesssim (1+t)^{-\frac{n}{2\sigma}(\frac{1}{m}-\frac{1}{2})} \|u_1\|_{L^m \cap L^2}.
\end{align*}
Furthermore, if we suppose that the following conditions are provided for $1 \leq \alpha_1 \leq m \leq \alpha_2 \leq \infty$:
$$ n\left(\frac{1}{\alpha_1}-\frac{1}{\alpha_2}\right)+n\sigma \max\left\{\frac{1}{2}-\frac{1}{\alpha_1}, \frac{1}{\alpha_2}-\frac{1}{2}\right\} < \sigma $$
and
$$ \frac{1}{2} \le \frac{1}{m}-\frac{1}{\alpha_2} < \frac{2\sigma}{n}, $$
then we conclude the estimate for the solution itself as follows:
\begin{align*}
    \|u(t,\cdot)\|_{L^{\alpha_2}} \lesssim (1+t)^{-\frac{n}{\sigma}(\frac{1}{\alpha_1}-\frac{1}{\alpha_2})+1}\|u_1\|_{L^{\alpha_1}} + e^{-ct} \|u_1\|_{L^m},
\end{align*}
where $c$ is a suitable positive constant.
\end{proposition}
\begin{remark}\label{remark2.1}
\fontshape{n}
\selectfont
    To prove Theorem \ref{Theorem1}, we need to use some estimates by choosing the parameters $m =1$, $\alpha_1 = 1$ and $\alpha_2 \in \{q,\infty\}$ with $\sigma > 1$. It is clear to verify that from the assumptions of Theorem \ref{Theorem1} the conditions appearing in Proposition \ref{Proposition2.1} are satisfied.
\end{remark}

\begin{remark}\label{remark2.2}
\fontshape{n}
\selectfont
    This remark gives some explanation for the restriction $\sigma\ne 1$ in Proposition \ref{Proposition2.1}. Actually, to establish the desired estimates in this proposition, the authors in \cite{DabbiccoEbert2021} have relied on the method of estimating Fourier multipliers appearing in the representation formula of fundamental solutions to \eqref{Problem1}. Unfortunately, this approach fails in the treatment of the case $\sigma=1$ because several properties of hyperbolic operators are missing in the case of so-called strictly hyperbolic operators, i.e. in the case $\sigma=1$. Additionally, another reason comes from the application of Littman's lemma, based on stationary phase method, which only works so well if the non-singular behavior of the Hessian matrix is guaranteed. However, the Hessian matrix is just singular in the case $\sigma=1$. Concerning this case we refer the readers to Section \ref{Finalcomments}.
\end{remark}
Next, we consider the following linear Cauchy problem:
\begin{equation}\label{Problem2}
    \begin{cases}
        v_{tt}+ (-\Delta)^{\sigma} v + v_t = 0, & x \in \mathbb{R}^n, t > 0,\\
        v(0,x) = 0, \quad v_t(0,x) = v_1(x), & x \in \mathbb{R}^n,
    \end{cases}
\end{equation}
where $\sigma \geq 1$. The solution of the problem (\ref{Problem2}) is expressed by
\begin{align}
    v^{\rm lin}(t,x) := \mathfrak{F}^{-1}_{\xi \to x}\left(\widehat{\mathcal{K}_{2}}(t,\xi)  \widehat{v_1}(\xi)\right), \label{solution2}
\end{align}
where
\begin{align*}
    \widehat{\mathcal{K}_2}(t,\xi) = \mathfrak{F}_{x \to \xi}(\mathcal{K}_2(t,x)) =\frac{e^{\lambda_{21}(\xi) t}-e^{\lambda_{22}(\xi) t}}{\lambda_{21}(\xi)-\lambda_{22}(\xi)},
\end{align*}
in which the characteristic roots are determined as follows:
\begin{align*}
    \lambda_{21,22} :=
    \begin{cases}
        -\displaystyle\frac{1}{2} \pm \frac{1}{2} \displaystyle\sqrt{1-4|\xi|^{2\sigma}} &\text{ if } |\xi| < 2^{-\frac{1}{\sigma}},\\
       \\
       -\displaystyle\frac{1}{2} \pm \frac{i}{2} \displaystyle\sqrt{4|\xi|^{2\sigma}-1} &\text{ if } |\xi| \geq 2^{-\frac{1}{\sigma}}.
    \end{cases}
\end{align*}

\begin{proposition}[see Proposition 2.1 in \cite{PhamReissig2017}]\label{Proposition2.2}
   Let $\sigma\ge 1$. The Sobolev solutions to \eqref{Problem2} satisfy the $(L^1 \cap L^2)-L^2$ estimates
\begin{align*}
\|v(t,\cdot)\|_{L^2} &\lesssim  (1+t)^{-\frac{n}{4\sigma}}\|v_1\|_{L^1 \cap H^{-{\sigma}}}, \\
\big\||D|^{\sigma} v(t,\cdot)\big\|_{L^2} &\lesssim  (1+t)^{-\frac{n}{4\sigma}- \frac{1}{2}}\|v_1\|_{L^1 \cap L^2}, \\
\|v_t(t,\cdot)\|_{L^2} &\lesssim (1+t)^{-\frac{n}{4\sigma}-1}\|v_1\|_{L^1 \cap L^2},
\end{align*}
and the $L^2-L^2$ estimates
\begin{align*}
\|v(t,\cdot)\|_{L^2} &\lesssim (1+t)\|v_1\|_{L^2}, \\
\big\||D|^{\sigma} v(t,\cdot)\big\|_{L^2} &\lesssim  (1+t)^{-\frac{1}{2}}\|v_1\|_{L^2}, \\
\|v_t(t,\cdot)\|_{L^2} &\lesssim \|v_1\|_{L^2},
\end{align*}
for all space dimensions $n \ge 1$.
\end{proposition}

\subsection{Philosophy of our approach}\label{subsection2.2}
By Duhamel's principle, we shall  interpret solutions to (\ref{Main.Eq.1}) as solutions of the following coupled system of non-linear integral equations:
$$\begin{cases}
u(t,x)= u^{\text{lin}}(t,x) + \displaystyle\int_0^t \mathcal{K}_{1}(t-\tau,x) \ast_x |v(\tau,x)|^p d\tau=: u^{\text{lin}}(t,x)+ u^{\text{non}}(t,x), \\
\\
v(t,x)= v^{\text{lin}}(t,x) + \displaystyle\int_0^t \mathcal{K}_{2}(t-\tau,x) \ast_x |u(\tau,x)|^q d\tau=: v^{\text{lin}}(t,x)+ v^{\text{non}}(t,x).
\end{cases}$$
Here the terms $u^{\rm lin}(t,x)$ and $v^{\rm lin}(t,x)$ are defined in (\ref{solution1}) and (\ref{solution2}), respectively. We introduce the family $\{X(t)\}_{t>0}$ of the solution spaces
$$ X(t):= \Big(\mathcal{C}\big([0,t],H^{\sigma} \cap L^{\infty} \cap L^q \big)\cap \mathcal{C}^1\big([0,t],L^2\big)\Big) \times \Big(\mathcal{C}\big([0,t],H^{\sigma}\big)\cap \mathcal{C}^1\big([0,t],L^2\big)\Big) $$
with the norm
\begin{align*}
\|(u,v)\|_{X(t)}&:= \sup_{0\le \tau \le t} \Big( f_1(\tau)^{-1}\|u(\tau,\cdot)\|_{L^q} + f_2(\tau)^{-1} \|u(\tau,\cdot)\|_{L^{\infty}} \\
&\hspace{2cm} +f_3(\tau)^{-1}\big\||D|^{\sigma} u(\tau,\cdot)\big\|_{L^2}+f_3(\tau)^{-1}\big\|u_t(\tau,\cdot)\big\|_{L^2} \\
&\hspace{2cm} +g_1(\tau)^{-1}\|v(\tau,\cdot)\|_{L^2} + g_2(\tau)^{-1}\big\||D|^{\sigma} v(\tau,\cdot)\big\|_{L^2}+ g_3(\tau)^{-1}\|v_t(\tau,\cdot)\|_{L^2} \Big),
\end{align*}
where
\begin{align*}
f_1(\tau)= (1+\tau)^{-\frac{n}{\sigma}(1-\frac{1}{q})+ 1+ [\varepsilon_1(p)]^{+}},\quad f_2(\tau)= (1+\tau)^{-\frac{n}{\sigma}+ 1+[\varepsilon_1(p)]^{+}}, \quad f_3(\tau) = (1+\tau)^{-\frac{n}{4\sigma}+ [\varepsilon_1(p)]^{+}},
\end{align*}
and
\begin{align*}
g_1(\tau)= (1+\tau)^{-\frac{n}{4\sigma}+[\varepsilon_2(q)]^{+}},\quad g_2(\tau)= (1+\tau)^{-\frac{n}{4\sigma}-\frac{1}{2}+[\varepsilon_2(q)]^{+}}, \quad g_3(\tau)=(1+\tau)^{-\frac{n}{2\sigma}+[\varepsilon_2(q)]^{+}}.
\end{align*}
We define the following operator for all $t>0$:
\begin{align*}
\mathcal{N}: \,\, (u,v) \in X(t) \longrightarrow \mathcal{N}[u,v](t,x) &= \big(u^{\text{lin}}(t,x), v^{\text{lin}}(t,x)\big)+ (u^{\text{non}}(t,x), v^{\text{non}}(t,x)\big).
\end{align*}
We shall indicate that the operator $\mathcal{N}$ fulfills the following two inequalities:
\begin{align}
\|\mathcal{N}[u,v]\|_{X(t)} &\lesssim \|(u_1,v_1)\|_{\mathcal{D}}+ \|(u,v)\|^p_{X(t)}+ \|(u,v)\|^q_{X(t)}, \label{pt4.3} \\
\|\mathcal{N}[u,v]- \mathcal{N}[\bar{u},\bar{v}]\|_{X(t)} &\lesssim \|(u,v)-(\bar{u},\bar{v})\|_{X(t)} \Big(\|(u,v)\|^{p-1}_{X(t)}+ \|(\bar{u},\bar{v})\|^{p-1}_{X(t)} \nonumber \\
&\hspace{6cm}+ \|(u,v)\|^{q-1}_{X(t)}+ \|(\bar{u},\bar{v})\|^{q-1}_{X(t)}\Big). \label{pt4.4}
\end{align}
Then, we may conclude global (in time) existence results of small data solutions by applying Banach's fixed point theorem. To do this, the following proposition from Harmonic Analysis comes into play in our proof.

\begin{proposition}[Fractional Gagliardo-Nirenberg inequality, see \cite{DaoReissig1, DaoReissig2}] \label{FractionalG-N}
Let $1<q,\, q_1,\, q_2<\infty$, $a >0$ and $s\in [0,a)$. Then, it holds
$$ \|u\|_{\dot{H}^{s}_q}\lesssim \|u\|_{L^{q_1}}^{1-\theta}\,\, \|u\|_{\dot{H}^a_{q_2}}^\theta, $$
where
$$ \theta=\theta_{s,a}(q,q_1,q_2)=\frac{\frac{1}{q_1}-\frac{1}{q}+\frac{s}{n}}{\frac{1}{q_1}-\frac{1}{q_2}+\frac{a}{n}}\quad \text{ and }\quad \frac{s}{a}\leq \theta\leq 1. $$
\end{proposition}
\subsection{Proof of Theorem \ref{Theorem1}} Let us divide our consideration into three cases as follows:

\noindent \textbf{First, let us assume that $q \leq q_{\rm crit}$ and $p > p_{\rm crit}$}. Thus, this implies that
 \begin{align*}
    \displaystyle\frac{pq+p+1}{2pq-p-1} = \displaystyle \max\left\{\frac{2q+1}{pq +q-2},  \frac{pq+p+1}{2pq-p-1} \right\},
 \end{align*}
 combining with the relation (\ref{relation10}), we choose $[\varepsilon_1(p)]^+ = 0$ and $[\varepsilon_2(q)]^+ = \varepsilon_2(q)$. From Proposition \ref{Proposition2.1} and \ref{Proposition2.2} we have the conclusion
\begin{align*}
    \|(u^{\rm lin}, v^{\rm lin})\|_{X(t)} \lesssim \|(u_1,v_1)\|_{\mathcal{D}}.
\end{align*}
For this reason, to prove the estimate (\ref{pt4.3}) we only need to prove the following estimate:
\begin{align}\label{Es1}
    \|(u^{\rm non}, v^{\rm non})\|_{X(t)} \lesssim \|(u,v)\|^p_{X(t)}+ \|(u,v)\|^q_{X(t)}.
\end{align}
In the first step we estimate the norm $\|u^{\rm non}(t,\cdot)\|_{L^{\alpha}}$ with $\alpha \in \{q, \infty\}$ in the following way:
\begin{align*}
    \|u^{\rm non}(t,\cdot)\|_{L^{\alpha}} \lesssim& \int_0^t (1+t-\tau)^{-\frac{n}{\sigma}(1-\frac{1}{\alpha})+1} \||v(\tau,\cdot)|^p\|_{L^1} d\tau.
\end{align*}
Applying the fractional Gagliardo-Nirenberg inequality from Proposition \ref{FractionalG-N} we can conclude with $\sigma < n < 2\sigma$ that
\begin{align}
    \||v(\tau,\cdot)|^p\|_{L^1} = \|v(\tau,\cdot)\|_{L^p}^p &\lesssim (1+\tau)^{-\frac{n}{2\sigma}(p-1)+p\varepsilon_2(q)} \|(u,v)\|_{X(t)}^{p}, \label{Estimate1}\\
    \||v(\tau,\cdot)|^p\|_{L^2} = \|v(\tau,\cdot)\|_{L^{2p}}^p &\lesssim (1+\tau)^{-\frac{n}{2\sigma}(p-\frac{1}{2})+p\varepsilon_2(q)} \|(u,v)\|_{X(t)}^{p} \label{Estimate2}.
\end{align}
The condition (\ref{condition1.1.1}) leads to
\begin{align}
    -\frac{n}{2\sigma}(p-1)+p\varepsilon_2(q) < -1, \label{re1}
\end{align}
associated with relation $-\frac{n}{\sigma} +1 > -1$, so we obtain
\begin{align*}
    \int_0^t (1+t-\tau)^{-\frac{n}{\sigma}(1-\frac{1}{\alpha})+1} \||v(\tau,\cdot)|^p\|_{L^1} d\tau \lesssim (1+t)^{-\frac{n}{\sigma}(1-\frac{1}{\alpha})+1} \|(u,v)\|_{X(t)}^{p}.
\end{align*}
From this, one may achieve
\begin{align}\label{Main Estimate 1}
    \|u^{\rm non}(t,\cdot)\|_{L^{\alpha}} \lesssim (1+t)^{-\frac{n}{\sigma}(1-\frac{1}{\alpha})+1} \|(u,v)\|_{X(t)}^{p} \text{ for } \alpha \in \{q, \infty\}.
\end{align}
In the second step we control the norms $\||D|^{\sigma}u^{\rm non}(t,\cdot)\|_{L^{2}}$ and $\|u_t^{\rm non}(t,\cdot)\|_{L^{2}}$ as follows:
\begin{align*}
    \||D|^{\sigma} u^{\rm non}(t,\cdot)\|_{L^2} + \|\partial_t u^{\rm non}(t,\cdot)\|_{L^2} &\lesssim \int_0^{t/2}(1+t-\tau)^{-\frac{n}{4\sigma}} \||v(\tau,\cdot)|^p\|_{L^1\cap L^2} d\tau \\
    &\quad +\int_{t/2}^t \| |v(\tau,\cdot)|^p\|_{L^2} d\tau \\
    &=: I_1(t) + I_2(t).
\end{align*}
Using the estimates (\ref{Estimate1}), (\ref{Estimate2}) and the relation (\ref{re1}) we derive
\begin{align*}
    I_1(t) &\lesssim (1+t)^{-\frac{n}{4\sigma}} \|(u,v)\|_{X(t)}^p\int_0^{t/2} (1+\tau)^{-\frac{n}{2\sigma}(p-1)+p\varepsilon_2(q)} d\tau \lesssim (1+t)^{-\frac{n}{4\sigma}} \|(u,v)\|_{X(t)}^p,\\
    I_2(t) &\lesssim  (1+t)^{ -\frac{n}{2\sigma}(p-\frac{1}{2})+1+p\varepsilon_2(q)} \|(u,v)\|_{X(t)}^p\lesssim (1+t)^{-\frac{n}{4\sigma}} \|(u,v)\|_{X(t)}^p.
\end{align*}
Therefore, we have established the following estimate:
\begin{align}\label{Main Estimate 2}
     \||D|^{\sigma} u^{\rm non}(t,\cdot)\|_{L^2} + \|u_t^{\rm non}(t,\cdot)\|_{L^2} \lesssim (1+t)^{-\frac{n}{4\sigma}} \|(u,v)\|_{X(t)}^p.
\end{align}
Next, we will estimate the norms $\|v^{\rm non}(t,\cdot)\|_{L^2}$, $\||D|^{\sigma}v^{\rm non}(t,\cdot)\|_{L^2}$ and $\|v_t^{\rm non}(t,\cdot)\|_{L^2}$. Using again the $(L^1 \cap L^2)-L^2$ estimates if $\tau \in [0, t/2]$ and the $L^2 - L^2$ estimates if $\tau \in [t/2, t]$ from Proposition \ref{Proposition2.2} for  $\||D|^{\sigma}v^{\rm non}(t,\cdot)\|_{L^2}$ and $\|v_t^{\rm non}(t,\cdot)\|_{L^2}$, we derive
\begin{align*}
    \|v^{\rm non}(t,\cdot)\|_{L^2} \lesssim& \int_0^t (1+t-\tau)^{-\frac{n}{4\sigma}}\||u(\tau,\cdot)|^q\|_{L^1 \cap L^2} d\tau,\\
    \||D|^{\sigma} v^{\rm non}(t,\cdot)\|_{L^2} \lesssim& \int_0^{t/2} (1+t-\tau)^{-\frac{n}{4\sigma}-\frac{1}{2}} \| |u(\tau,\cdot)|^q\|_{L^1 \cap L^2} d\tau\\
    &+ \int_{t/2}^t (1+t-\tau)^{-\frac{1}{2}} \| |u(\tau,\cdot)|^q \|_{L^2} d\tau,\\
    \|v_t^{\rm non}(t,\cdot)\|_{L^2} \lesssim& \int_0^{t/2} (1+t-\tau)^{-\frac{n}{4\sigma}-1} \| |u(\tau,\cdot)|^q\|_{L^1 \cap L^2} d\tau\\
    &+ \int_{t/2}^t \| |u(\tau,\cdot)|^q \|_{L^2} d\tau.
\end{align*}
From the definition of the norm in the solution space $X(t)$, we have the following estimates:
\begin{align*}
    \||u(\tau,\cdot)|^q\|_{L^1} = \|u(\tau,\cdot)\|_{L^q}^q &\lesssim (1+\tau)^{q-\frac{n}{\sigma}(q-1)} \|(u,v)\|_{X(t)}^q,\\
    \|u(\tau,\cdot)\|_{L^{\infty}}^q &\lesssim (1+\tau)^{q-\frac{nq}{\sigma}} \|(u,v)\|_{X(t)}^q.
\end{align*}
Using the fractional Gagliardo-Nirenberg inequality from Proposition \ref{FractionalG-N} we obtain
\begin{align*}
\||u(\tau,\cdot)|^q\|_{L^2}=\|u(\tau,\cdot)\|_{L^{2q}}^q \lesssim (1+\tau)^{q-\frac{n}{\sigma}(q-\frac{1}{2})} \|(u,v)\|_{X(t)}^q.
\end{align*}
As a consequence, we get
\begin{align*}
    \|v^{\rm non}(t,\cdot)\|_{L^2} \lesssim& (1+t)^{-\frac{n}{4\sigma}}\|(u,v)\|_{X(t)}^q \int_{0}^{t/2}(1+\tau)^{q-\frac{n}{\sigma}(q-1)}d\tau\\
    &+ (1+t)^{q-\frac{n}{\sigma}(q-1)}\|(u,v)\|_{X(t)}^q\int_{t/2}^t (1+t-\tau)^{-\frac{n}{4\sigma}} d\tau\\
    \lesssim & (1+t)^{-\frac{n}{4\sigma}+[\varepsilon_2(q)]^{+}} \|(u,v)\|_{X(t)}^q,\\
    \||D|^{\sigma} v^{\rm non}(t,\cdot)\|_{L^2} \lesssim& (1+t)^{-\frac{n}{4\sigma}-\frac{1}{2}} \|(u,v)\|_{X(t)}^q \int_{0}^{t/2}(1+\tau)^{q-\frac{n}{\sigma}(q-1)}d\tau\\
    &+ \|(u,v)\|_{X(t)}^q \int_{t/2}^t (1+t-\tau)^{-\frac{1}{2}} (1+\tau)^{q-\frac{n}{\sigma}(q-\frac{1}{2})} d\tau \\
    \lesssim& (1+t)^{-\frac{n}{4\sigma}-\frac{1}{2}+[\varepsilon_2(q)]^{+}} \|(u,v)\|_{X(t)}^q,\\
    \|v_t^{\rm non}(t,\cdot)\|_{L^2} \lesssim& (1+t)^{-\frac{n}{4\sigma}-1}\|(u,v)\|_{X(t)}^q\int_0^{t/2} (1+\tau)^{q-\frac{n}{\sigma}(q-1)} d\tau\\
    &+ \|(u,v)\|_{X(t)}^q\int_{t/2}^t (1+\tau)^{q-\frac{n}{\sigma}(q-\frac{1}{2})}d\tau\\
    \lesssim& (1+t)^{-\frac{n}{2\sigma}+[\varepsilon_2(q)]^{+}} \|(u,v)\|_{X(t)}^q.
\end{align*}
Hence, we link the previous estimates to (\ref{Main Estimate 1}) and (\ref{Main Estimate 2}) to get the estimate (\ref{Es1}). By the same way, let us now sketch briefly the proof of the estimate (\ref{pt4.4}). Taking two functions $(u, v)$ and $(\bar{u}, \bar{v})$ in the solution space $X(t)$ one notices that
\begin{align*}
    \mathcal{N}[u,v](t,x)-\mathcal{N}[\bar{u}, \bar{v}](t,x) = \left(u^{\rm non}(t,x)-\bar{u}^{\rm non}(t,x), v^{\rm non}(t,x)-\bar{v}^{\rm non}(t,x)\right).
\end{align*}
Then, applying H\"older's inequality and the fractional Gagliardo-Nirenberg inequality from Proposition \ref{FractionalG-N} we arrive at
\begin{align*}
    &\| |u(\tau,\cdot)|^{q}-|\bar{u}(\tau,\cdot)|^{q} \|_{L^{\gamma}} \\
    &\qquad \lesssim \|(u-\bar{u})(\tau,\cdot)\|_{L^{q\gamma}} \left(\|u(\tau,\cdot)\|_{L^{q\gamma}}^{q-1}+\|\bar{u}(\tau,\cdot)\|_{L^{q\gamma}}^{q-1}\right) \\
    &\qquad \lesssim (1+\tau)^{q-\frac{n}{\sigma}(q-\frac{1}{\gamma})} \|(u,v)-(\bar{u}, \bar{v})\|_{X(t)} \left(\|(u,v)\|_{X(t)}^{q-1}+ \|(\bar{u},\bar{v})\|_{X(t)}^{q-1} \right),\\
    &\| |v(\tau,\cdot)|^{p}-|\bar{v}(\tau,\cdot)|^{p}\|_{L^{\gamma}} \\
    &\qquad \lesssim \|(v-\bar{v})(\tau,\cdot)\|_{L^{p\gamma}} \left(\|v(\tau,\cdot)\|_{L^{p\gamma}}^{p-1}+\|\bar{v}(\tau,\cdot)\|_{L^{p\gamma}}^{p-1}\right) \\
    &\qquad \lesssim (1+\tau)^{-\frac{n}{2\sigma}(p-\frac{1}{\gamma})+p[\varepsilon_2(q)]^{+}} \|(u,v)-(\bar{u},\bar{v})\|_{X(t)}\left(\|(u,v)\|_{X(t)}^{p-1}+\|(\bar{u},\bar{v})\|_{X(t)}^{p-1}\right),
\end{align*}
where $\gamma = 1,2$. Then, performing similar steps of the proof as above, we gain the estimate (\ref{pt4.4}). All in all, Theorem \ref{Theorem1} has been proved. \medskip

\noindent \textbf{Next, let us consider the case $q > q_{\rm crit}$ and $p \leq p_{\rm crit}$}. This yields that
\begin{align*}
     \max\left\{\frac{pq+p+1}{2pq-p-1},  \frac{2q+1}{pq+q-2} \right\}= \frac{2q+1}{pq+q-2}.
\end{align*}
Observing from the condition (\ref{condition1.1.1}) we choose $[\varepsilon_1(p)]^{+} = \varepsilon_1(p)$, $[\varepsilon_2(q)]^{+} = 0$. Combining this with the relation (\ref{relation10}) gives
 \begin{align*}
    q-\frac{n}{\sigma}(q-1)+q[\varepsilon_1(p)]^{+} < -1.
\end{align*}
Then, following some steps as in the first case we can conclude Theorem \ref{Theorem1} in this case. \medskip

\noindent \textbf{Finally, let us consider the case $q > q_{\rm crit}$ and $p > p_{\rm crit}$}. Noticing that from the condition (\ref{condition1.1.1}) we choose $[\varepsilon_1(p)]^{+}= [\varepsilon_2(q)]^{+} = 0$ and repeat some steps as in the first case, we may finish the proof of Theorem \ref{Theorem1} in this case. Summarizing, the proof of Theorem \ref{Theorem1} is complete.

\section{Blow-up}\label{section4}
\subsection{Preliminaries}\label{Section3.1}
Before giving our proof, let us introduce the definition of weak solutions for \eqref{Main.Eq.1} together with some auxiliary tools.
\begin{definition}\label{Definition4.1}
    Let $p, q > 1, j = 1,2$ and $T \in (0, \infty]$. We say that
    \begin{align*}
        (u,v) \in \mathcal{C}([0,T), L^q(\mathbb{R}^n)) \times \mathcal{C}([0,T), L^2(\mathbb{R}^n))
    \end{align*}
    is a local (in time) weak solution to \eqref{Main.Eq.1}  if for any test functions $\Phi_{\rm j} = \Phi_{\rm j}(t,x): = \eta(t)\varphi_{\rm j}(x) $ with $\eta  \in \mathcal{C}_0^{\infty}[0,T)$, and $\varphi_{\rm j} \in \mathcal{C}^{\infty}(\mathbb{R}^n)$ satisfying that all their derivatives in $L^1(\mathbb{R}^n) \cap L^{\infty}(\mathbb{R}^n)$ the following relations hold:
\begin{align*}
       &\displaystyle\int_0^{T} \int_{\mathbb{R}^n} |v(t,x)|^p \Phi_{2}(t,x)dx dt + \int_{\mathbb{R}^n} u_1(x) \Phi_2(0,x) dx \\ &\qquad\quad\quad\quad\quad\quad\quad\quad =\displaystyle\int_0^T \int_{\mathbb{R}^n} u(t,x) \left(\partial_t^2 - \partial_t (-\Delta)^{\sigma} + (-\Delta)^{\sigma}\right)\Phi_2(t,x) dx dt,
\end{align*}
and
\begin{align*}
        &\displaystyle\int_0^{T} \int_{\mathbb{R}^n} |u(t,x)|^q \Phi_{1}(t,x)dx dt + \int_{\mathbb{R}^n} v_1(x) \Phi_1(0,x) dx \\ &\qquad\quad\quad\quad\quad\quad\quad\quad = \displaystyle\int_0^T \int_{\mathbb{R}^n} v(t,x) \left(\partial_t^2 - \partial_t + (-\Delta)^{\sigma}\right)\Phi_1(t,x) dx dt.
\end{align*}
If $T = \infty$, then we say that $(u,v)$ is a global (in time) weak solution to \eqref{Main.Eq.1}.
\end{definition}
Next, the existence of local (in time) solutions to \eqref{Main.Eq.1} can be verified by the following proposition (see, moreover, \cite{CazenaveHaraux,KainaneReissig2023,NishiharaWakasugi2014}).

\begin{proposition} \label{LocalExistence}
Assume that $1<\sigma<n< 2\sigma$ together with two exponents $p,q\ge 2$. Then, there exists a positive number $T=T(u_1,v_1)$ such that for any initial data $ (u_1,v_1)\in \mathcal{D}$ we have a unique local (in time) Sobolev solution
	$$ (u,v)\in \Big(\mathcal{C}\big([0,T],H^{\sigma}\big)\cap \mathcal{C}^{1}\big([0,T],L^{2}\big)\Big)^2 $$
to \eqref{Main.Eq.1}.
\end{proposition}
\begin{proof}
We will follow the same approach as we did in the proof of Theorem \ref{Theorem1} to verify Proposition \ref{LocalExistence}. For this reason, let us briefly sketch the proof of this proposition only. Indeed, we introduce the solution space by
$$
X(t):= \Big(\mathcal{C}\big([0,t],H^{\sigma}\big)\cap \mathcal{C}^{1}\big([0,t],L^{2}\big)\Big) \times \Big(\mathcal{C}\big([0,t],H^{\sigma}\big)\cap \mathcal{C}^{1}\big([0,t],L^{2}\big)\Big)
$$
together with the norm
$$ \|(u,v)\|_{X(t)}=\sup_{0\leq \tau\leq t}\big\{\Gamma(\tau,u)+ \Gamma(\tau,v)\big\}, $$
where
\begin{align*}
\Gamma(\tau,w) &:= \|w(\tau,\cdot)\|_{L^{2}} + \big\| |D|^{\sigma}w(\tau,\cdot)\big\|_{L^{2}}+ \big\|w_{t}(\tau,\cdot)\big\|_{L^{2}}
\end{align*}
with $w=u$ or $w=v$. We define the operator $\mathcal{N}$ for any $(u,v)\in X(T)$ as follows:
$$ \mathcal{N}(u,v)(t,x)= \big(u^{\rm lin}(t,x)+u^{\rm non}(t,x), v^{\rm lin}(t,x)+v^{\rm non}(t,x)\big), $$
where $w^{\rm lin}(t,x)$ and $w^{\rm non}(t,x)$, with $w=u$ or $w=v$, are defined as in the proof of Theorem \ref{Theorem1}. Again, by the aid of the fractional Gagliardo-Nirenberg inequality we may conclude that
\begin{align*}
\|\mathcal{N}(u,v)\|_{X(t)} &\le  C_0 \|(u_1,v_1)\|_{\mathcal{D}} +C_1(t)\|(u,v)\|_{X(t)}^{p}+C_2(t)\|(u,v)\|_{X(t)}^{q}, \\
\|\mathcal{N}(u,v)-\mathcal{N}(\bar{u},\bar{v})\|_{X(t)} &\le C_3(t)\|(u,v)-(\bar{u},\bar{v})\|_{X(t)} \Big( \|(u,v)\|_{X(t)}^{p-1}+\|(u,v)\|_{X(t)}^{q-1} \nonumber \\
&\hspace{6cm} +\|(\bar{u},\bar{v})\|_{X(t)}^{p-1}+\|(\bar{u},\bar{v})\|_{X(t)}^{q-1}\Big).
\end{align*}
Here $C_0$ is a non-negative constant and $C_k(t)$ with $k=1,2,3$ are time-dependent functions, provided that $C_k(t)$ with $k=1,2,3$ tend to $0$ as $t$ tends to $0$. In this way, the last two inequalities enable us to employ Banach's fixed point theorem to indicate the local (in time) existence of large data solutions to \eqref{Main.Eq.1}. To be more specific, we introduce
$$ M_0:= C_0 \|(u_1,v_1)\|_{\mathcal{D}} $$
and assume that $\|(u,v)\|_{X(t)}\le 2M_0$, $\|(\bar{u},\bar{v})\|_{X(t)}\le 2M_0$ for a suitable value of $t$. Thus, from the above two inequalities it follows immediately
\begin{align*}
\|\mathcal{N}(u,v)\|_{X(t)} &\le  M_0+C_1(t)(2M_0)^{p}+C_2(t)(2M_0)^{q}, \\
\|\mathcal{N}(u,v)-\mathcal{N}(\bar{u},\bar{v})\|_{X(t)} &\le C_3(t)\Big(2(2M_0)^{p-1}+2(2M_0)^{q-1}\Big)\|(u,v)-(\bar{u},\bar{v})\|_{X(t)}.
\end{align*}
Now let us choose $t$ suitably to fulfill the following conditions:
$$ C_1(t)(2M_0)^{p}\le \frac{M_0}{2},\quad C_2(t)(2M_0)^{q}\le \frac{M_0}{2}\quad\text{ and }\quad  C_3(t)\Big(2(2M_0)^{p-1}+2(2M_0)^{q-1}\Big)\le \frac{1}{2}.  $$
As a consequence, all these estimates lead to
$$ \|\mathcal{N}(u,v)\|_{X(t)}\le  2M_0 \quad \text{ and }\quad  \|\mathcal{N}(u,v)-\mathcal{N}(\bar{u},\bar{v})\|_{X(t)} \le \frac{1}{2}\|(u,v)-(\bar{u},\bar{v})\|_{X(t)}, $$
which imply the desired statement in Proposition \ref{LocalExistence} thanks to Banach's fixed point theorem.
\end{proof}
Eventually, let us recall several known ingredients which will be used in the proof of Theorem \ref{Theorem3}.
\begin{lemma}[see Corollary 3.1 in \cite{DabbiccoFujiwara2021}] \label{lemma4.1}
Let $\gamma > 0$ and $s:= \gamma- [\gamma]$. Then, the following estimates hold for all $x \in \R^n$ and $r>n$:
$$ \big|(-\Delta)^\gamma \langle x \rangle^{-r}\big| \lesssim
\begin{cases}
\langle x \rangle^{-r-2\gamma} &\text{ if } \gamma \text{ is an integer},\\
\langle x \rangle^{-n-2s} &\text{ otherwise }.
\end{cases} $$
\end{lemma}

\begin{lemma}[see Lemma 4 in \cite{DaoReissig2021}] \label{Lemma4.2}
Let $s\in (0,1)$. Assume that $\psi$ is a smooth function satisfying $\partial_x^2 \psi\in L^\ity$. For any $R>0$, let $\psi_R$ be a function defined by
$$ \psi_R(x)= \psi\big(R^{-1} x\big)\quad \text{ for all }x \in \R^n. $$
Then, $(-\Delta)^s (\psi_R)$ enjoys the following scaling property for all $x \in \R^n$:
$$(-\Delta)^s (\psi_R)(x)= R^{-2s}\big((-\Delta)^s \psi \big)\big(R^{-1} x\big). $$
\end{lemma}

\subsection{Proof of Theorem \ref{Theorem3}}
At first, we introduce the test function $\eta= \eta(t)$:
\begin{align}
&1.\quad \eta \in \mathcal{C}_0^\ity \big([0,\ity)\big) \text{ and }
\eta(t)=\begin{cases}
1 & \text{ if }\, 0 \le t \le \frac{1}{2}, \\
\text{decreasing } &\text{ if } \frac{1}{2} < t < 1,\\
0 & \text{ if }\, t \ge 1,
\end{cases} \nonumber \\
&2.\quad \eta^{-\frac{\kappa'}{\kappa}}(t)\big(|\eta'(t)|^{\kappa'}+|\eta''(t)|^{\kappa'}\big) \le C \quad \text{ for any } t \in [1/2,1], \label{t13.1.1}
\end{align}
with $\kappa= p$ or $\kappa= q$, where $\kappa'$ is the conjugate of $\kappa$ and $C$ is a suitable positive constant. Next, we introduce the following radial space-dependent test function:
\begin{align*}
\varphi = \varphi(x) := \langle x \rangle^{-n-2\overline{\sigma}}=(1+|x|^2)^{-n/2-\overline{\sigma}},
\end{align*}
where $\overline{\sigma}$ is given by
$$ \overline{\sigma}= \begin{cases}
\text{an arbitrary constant }\epsilon \in (0,1) &\text{ if }\sigma \text{ is an integer number}, \\
\sigma- [\sigma] &\text{ if }\sigma \text{ is a fractional number}.
\end{cases}$$
Thus, it follows that $\overline{\sigma} \in (0,1)$. Let $R$ be a large parameter in $[0,\ity)$. For $j = 1,2$, we denote the two test functions
$$ \Psi_{j,R}(t,x):= \eta_R(t) \varphi_{j,R}(x), $$
 where $\eta_R(t):= \eta(R^{-2\sigma}t)$ , $\varphi_{j,R}(x):= \varphi(R^{-j}x)$, moreover, we define the functionals
\begin{align}
\mathcal{I}_{j,R} &:= \int_0^{\ity}\int_{\R^n}|v(t,x)|^p \Psi_{j,R}(t,x)\,dxdt= \int_0^{R^{2\sigma}}\int_{\R^n}|v(t,x)|^p \eta_R(t)\varphi_{j,R}(x)\,dxdt, \notag\\
\mathcal{J}_{j,R} &:=  \int_0^{\ity}\int_{\R^n}|u(t,x)|^q \Psi_{j,R}(t,x)\,dxdt= \int_0^{R^{2\sigma}}\int_{\R^n}|u(t,x)|^q \eta_R(t)\varphi_{j,R}(x)\,dxdt, \notag\\
\mathcal{I}_R &:= \int_0^{\ity}\int_{\R^n}|v(t,x)|^p\eta_{R}(t)\,dxdt= \int_0^{R^{2\sigma}}\int_{\R^n}|v(t,x)|^p \eta_R(t)\,dxdt,\notag\\
\mathcal{J}_{R} &:=  \int_0^{\ity}\int_{\R^n}|u(t,x)|^q \eta_R(t)\,dxdt= \int_0^{R^{2\sigma}}\int_{\R^n}|u(t,x)|^q \eta_R(t)\,dxdt. \notag
\end{align}
Since $\eta_R=\eta_R(t)$ and
$\varphi_{j,R}=\varphi_{j,R}(x)$ are increasing functions with respect to $R$, it is obvious to see that $\mathcal{I}_{j,R},\, \mathcal{I}_{R}$,  $\mathcal{J}_{j,R}$ and $\mathcal{J}_R$ also increase with respect to $R$.
 Let us assume that $(u,v)= \big(u(t,x),v(t,x)\big)$ is a global (in time) weak solution belonging to the class
 $$(\mathcal{C}([0,\infty), L^q(\mathbb{R}^n))\cap L^q([0, \infty) \times \mathbb{R}^n)) \times \mathcal{C}([0,\infty), L^2(\mathbb{R}^n))$$
 to (\ref{Main.Eq.1}). We use Definition \ref{Definition4.1} in place of $\Phi_{j}(t,x)= \Psi_{j,R}(t,x)$ with $j=1,2$ to obtain
\begin{align}
&\mathcal{I}_{2,R} +\int_{\R^n} u_1(x)\varphi_{2,R}(x)\,dx \nonumber \\
&\qquad = \int_0^{R^{2\sigma}}\int_{\mathbb{R}^n} u(t,x)\left(\partial_t^2 -\partial_t (-\Delta)^{\sigma} +(-\Delta)^{\sigma}\right)\Psi_{2,R}(t,x)dxdt,\label{relation1}
\end{align}
and
\begin{align}
&\mathcal{J}_{1,R} + \int_{\R^n} v_1(x)\varphi_{1,R}(x)\,dx \nonumber \\
&\qquad = \int_0^{R^{2\sigma}}\int_{\mathbb{R}^n} v(t,x)\left(\partial_t^2 -\partial_t+(-\Delta)^{\sigma}\right)\Psi_{1,R}(t,x)dxdt. \label{relation2}
\end{align}
After applying H\"{o}lder's inequality with $\frac{1}{q}+\frac{1}{q'}=1$, then carrying out the change of variables $\tilde{t}:= R^{-2\sigma}t$, $\tilde{x}:= R^{-2}x$ combined with the property (\ref{t13.1.1}), we may proceed as follows:
\begin{align*}
&\int_0^{R^{2\sigma}}\int_{\mathbb{R}^n} |u(t,x)| \left|\left(\partial_t^2 -\partial_t (-\Delta)^{\sigma} +(-\Delta)^{\sigma}\right)\Psi_{2,R}(t,x)\right| dxdt\\
&\quad\lesssim \left(\int_0^{R^{2\sigma}} \int_{\mathbb{R}^n} |u(t,x)|^q \Psi_{2,R}(t,x)dxdt\right)^{\frac{1}{q}}\\ &\hspace{2cm} \times \left(\int_0^{R^{2\sigma}}\int_{\mathbb{R}^n} \left|\left(\partial_t^2 -\partial_t (-\Delta)^{\sigma} +(-\Delta)^{\sigma}\right)\Psi_{2,R}(t,x)\right|^{q'} \Psi_{2,R}^{-\frac{q'}{q}}(t,x)dxdt\right)^{\frac{1}{q'}}\\
&\quad\lesssim \mathcal{J}_{2,R}^{\frac{1}{q}} \left(2R^{-4\sigma}+R^{-6\sigma}\right) \left(\int_0^{R^{2\sigma}}\int_{\mathbb{R}^n} \varphi_{2,R}(x)dxdt \right)^{\frac{1}{q'}}\\
&\quad\lesssim \mathcal{J}_{2,R}^{\frac{1}{q}} R^{-4\sigma+\frac{2n+2\sigma}{q'}} \left(\int_0^1 \int_{\mathbb{R}^n} \varphi(\tilde{x}) d\tilde{x}d\tilde{t}\right) = C \mathcal{J}_{2,R}^{\frac{1}{q}} R^{-4\sigma+\frac{2n+2\sigma}{q'}} .
\end{align*}
From this and the relation (\ref{relation1}), we have the conclusion
\begin{align}\label{impor-ine1}
     \mathcal{I}_{2,R} + \int_{\R^n} u_1(x)\varphi_{2,R}(x)\,dx \lesssim   \mathcal{J}_{2,R}^{\frac{1}{q}} R^{-4\sigma+\frac{2n+2\sigma}{q'}} \text{ for all } R \geq 1.
\end{align}
In the same way, using the change of variables $\tilde{t}:= R^{-2\sigma}t$ and $\tilde{x}:= R^{-1} x$ we also get the following estimates:
\begin{align*}
&\int_0^{R^{2\sigma}}\int_{\mathbb{R}^n} |v(t,x)| \left|\left(\partial_t^2 -\partial_t+(-\Delta)^{\sigma}\right)\Psi_{1,R}(t,x)\right|dxdt\\
&\quad\lesssim \left(\int_0^{R^{2\sigma}} \int_{\mathbb{R}^n} |v(t,x)|^p \Psi_{1,R}(t,x) dxdt\right)^{\frac{1}{p}}\\
&\hspace{2cm} \times \left(\int_0^{R^{2\sigma}}\int_{\mathbb{R}^n} \left|\left(\partial_t^2 -\partial_t+(-\Delta)^{\sigma}\right)\Psi_{1,R}(t,x)\right|^{p'} \Psi_{1,R}^{-\frac{p'}{p}}(t,x) dxdt\right)^{\frac{1}{p'}}\\
&\quad\lesssim \mathcal{I}_{1,R}^{\frac{1}{p}} \left(R^{-4\sigma}+ 2R^{-2\sigma}\right) \left(\int_0^{R^{2\sigma}} \int_{\mathbb{R}^n}  \varphi_{1,R}(x) dxdt\right)^{\frac{1}{p'}}\\
&\quad\lesssim \mathcal{I}_{1,R}^{\frac{1}{p}} R^{-2\sigma+\frac{n+2\sigma}{p'}} \left(\int_0^1 \int_{\mathbb{R}^n}\varphi(\tilde{x}) d\tilde{x} d\tilde{t}\right)^{\frac{1}{p'}} = C \,\mathcal{I}_{1,R}^{\frac{1}{p}} R^{-2\sigma+\frac{n+2\sigma}{p'}}.
\end{align*}
Furthermore, combining this with the relation
 (\ref{relation2}) one arrives at
\begin{align}\label{impor-ine2}
    \mathcal{J}_{1,R} + \int_{\R^n} v_1(x)\varphi_{1,R}(x)\,dx \lesssim   \mathcal{I}_{1,R}^{\frac{1}{p}} R^{-2\sigma+\frac{n+2\sigma}{p'}} &\text{ for all } R \geq 1.
\end{align}
Due to $u \in L^q([0, \infty) \times \mathbb{R}^n)$, we apply the monotone convergence theorem to derive
\begin{align*}
    \lim_{R\to\infty} \mathcal{J}_{1,R} = \lim_{R\to\infty} \mathcal{J}_{2,R} = \lim_{R\to\infty}\mathcal{J}_R= \int_0^{\infty} \int_{\mathbb{R}^n} |u(t,x)|^q dxdt  < \infty .
\end{align*}
 Therefore there exists a sufficiently large positive number $R_0$ such that the following relations hold:
$$ \int_{\mathbb{R}^n} u_1(x) \varphi_{j,R}(x) dx > 0,\quad \int_{\mathbb{R}^n} v_1(x) \varphi_{j,R}(x) dx > 0 $$
and
$$ \mathcal{J}_{1,R} \leq \mathcal{J}_{2,R} \leq 2 \mathcal{J}_{1,R}, $$
for all $R \in \left[R_0, \infty\right)$ and $j \in \{1,2\}$. From the relation $\mathcal{I}_{1,R} \leq \mathcal{I}_{2,R}$, substituting the right-hand side of (\ref{impor-ine2})  into  the left-hand side of (\ref{impor-ine1}) we arrive at
  \begin{align}
      \mathcal{J}_{1,R} +\int_{\mathbb{R}^n} v_1(x) \varphi_{1,R}(x) dx &\lesssim \mathcal{J}_{2,R}^{\frac{1}{pq}} R^{-2\sigma-\frac{4\sigma}{p}+\frac{2n+2\sigma}{pq'}+\frac{n+2\sigma}{p'}} \notag\\
      &\lesssim \mathcal{J}_{1,R}^{\frac{1}{pq}} R^{-2\sigma-\frac{4\sigma}{p}+\frac{2n+2\sigma}{pq'}+\frac{n+2\sigma}{p'}}. \label{inequality1}
  \end{align}
Analogously, we substitute the right-hand side of (\ref{impor-ine1}) into the left-hand side of (\ref{impor-ine2}) to derive
  \begin{align}\label{inequality2}
      \mathcal{I}_{1,R}+ \int_{\mathbb{R}^n} u_1(x) \varphi_{2,R}(x) dx \lesssim \mathcal{I}_{1,R}^{\frac{1}{pq}} R^{-4\sigma-\frac{2\sigma}{q}+\frac{2n+2\sigma}{q'}+\frac{n+2\sigma}{p'q}}.
  \end{align}
Thus, it follows that
$$ \int_{\mathbb{R}^n} v_1(x) \varphi_{1,R}(x) dx \lesssim \mathcal{J}_{1,R}^{\frac{1}{pq}} R^{-2\sigma-\frac{4\sigma}{p}+\frac{2n+2\sigma}{pq'}+\frac{n+2\sigma}{p'}} - \mathcal{J}_{1,R} $$
and
$$ \int_{\mathbb{R}^n} u_1(x) \varphi_{2,R}(x) dx \lesssim \mathcal{I}_{1,R}^{\frac{1}{pq}} R^{-4\sigma-\frac{2\sigma}{q}+\frac{2n+2\sigma}{q'}+\frac{n+2\sigma}{p'q}} - \mathcal{I}_{1,R}, $$
  for all $R \in \left[R_0, \infty\right)$. From these observations, the application of the elementary inequality
  \begin{align*}
    Ay^{\beta} - y \lesssim A^{\frac{1}{1-\beta}} \text{ for any } A > 0, y \geq 0 \text{ and } 0 < \beta < 1,
\end{align*}
gives
  \begin{align}\label{ini1}
     \int_{\mathbb{R}^n} v_1(x) \varphi_{1,R}(x) dx \lesssim R^{\frac{pq}{pq-1}(-2\sigma-\frac{4\sigma}{p}+\frac{2n+2\sigma}{pq'}+\frac{n+2\sigma}{p'})},
  \end{align}
  and
  \begin{align}\label{ini2}
      \int_{\mathbb{R}^n} u_1(x) \varphi_{2,R}(x) dx \lesssim R^{\frac{pq}{pq-1}( -4\sigma-\frac{2\sigma}{q}+\frac{2n+2\sigma}{q'}+\frac{n+2\sigma}{p'q})}.
  \end{align}
Obviously, if the conditions (\ref{condition1.3.3}) or $n \leq \sigma $ occurs, then it implies that $\Gamma_{\rm c}(p,q) > 0$, where
\begin{equation} \label{CT1}\tag{$*$}
 \Gamma_{\rm c}(p,q):=\frac{pq}{pq-1}\max\left\{2\sigma+\frac{4\sigma}{p}-\frac{2n+2\sigma}{pq'}-\frac{n+2\sigma}{p'}, 4\sigma+\frac{2\sigma}{q}-\frac{2n+2\sigma}{q'}-\frac{n+2\sigma}{p'q}\right\}.
\end{equation}
Then, we take $R\to\infty$ in either the inequality (\ref{ini1}) or the inequality (\ref{ini2}) to establish a contradiction to (\ref{condition1.3.1}).
Summarizing, the proof of Theorem \ref{Theorem3} is completed.

\section{Final conclusions and open problems} \label{Sect.Final}
In this section, we will summarize how to get some estimates for lifespan of solutions in the subcritical case
\begin{equation}\label{Crit.Condition}
    \max\left\{\frac{2q+1}{pq+q-2},\frac{pq+p+1}{2pq-p-1}\right\} > \frac{n}{2\sigma}.
\end{equation}
Let us denote by $T_{\varepsilon}$ the so-called lifespan of a local (in time)
solution $(u,v)$, i.e the maximal existence time of local solutions. Moreover, we replace the initial data $(u_1,v_1)$ by $(\varepsilon u_1, \varepsilon v_1)$, in which $\varepsilon$ stands for a small, positive constant to describe the size of the initial data, and define the quantity $\Gamma_{\rm c}(p,q)$ as in (\ref{CT1}). Then, lower bound estimates and upper bound estimates for $T_{\varepsilon}$ are given by the next statements.

\subsection{Upper bound of lifespan} \label{Sectionupperbound}
\begin{proposition}\label{proposition4.1}
Under the same assumptions as in Theorem \ref{Theorem3} together with the condition \eqref{Crit.Condition} for the exponents $p$ and $q$, there exists a positive constant $\varepsilon_0>0$ such that for any $\varepsilon \in (0,\varepsilon_0]$ the lifespan of solutions to \eqref{Main.Eq.1} fulfill the following upper bound estimates:
    \begin{equation} \label{Upper_Lifespan}
    T_{\varepsilon} \leq
C\varepsilon^{-\frac{2\sigma}{\Gamma_{\rm c}(p,q)}},
    \end{equation}
    where $C$ is a suitable positive constant.
\end{proposition}
    \begin{proof}
     Let $R_1$ be a sufficiently large constant. Then the following relation holds:
    \begin{align*}
         \min\left\{\int_{\mathbb{R}^n} u_1(x)\varphi_{2,R}(x)dx,  \int_{\mathbb{R}^n} v_1(x)\varphi_{1,R}(x) dx \right\} > C_1 \text{ for all } R \geq R_1,
    \end{align*}
    where
    \begin{align*}
        C_1 := \frac{1}{2} \min\left\{\int_{\mathbb{R}^n}u_1(x)dx, \int_{\mathbb{R}^n} v_1(x) dx \right\}.
    \end{align*}
    Combining this with the relations (\ref{ini1}) and (\ref{ini2}) we arrive at
    \begin{align*}
        C_1 \varepsilon \leq C_2 R^{-\Gamma_{\rm c}(p,q)} = T_{\varepsilon}^{-\frac{\Gamma_{\rm c}(p,q)}{2\sigma}}, \text{ that is, } T_{\varepsilon} \leq C_2 C_1^{-1} \varepsilon^{-\frac{2\sigma}{\Gamma_{\rm c}(p,q)}} = C \varepsilon^{-\frac{2\sigma}{\Gamma_{\rm c}(p,q)}} .
    \end{align*}
    Summarizing, the proof of Proposition \ref{proposition4.1} is completed.
    \end{proof}

\subsection{Lower bound of lifespan} \label{Sectionlowerbound}
\begin{proposition}\label{proposition4.2}
    Let $1 < \sigma < n < 2\sigma$ and $(u_1,v_1) \in \mathcal{D}$. We assume that the exponents $p,q \ge 2$ satisfy the condition \eqref{Crit.Condition}. Then, there exists $\varepsilon_0 > 0$ such that for any $\varepsilon \in (0, \varepsilon_0]$, the lower bound for the lifespan of solutions to \eqref{Main.Eq.1} can be estimated as follows:
    \begin{equation} \label{Lower_Lifespan}
        T_{\varepsilon} \geq c \varepsilon^{-\frac{2\sigma}{\Gamma_{\rm c}(p,q)}},
    \end{equation}
    where $c$ is a suitable positive constant.
\end{proposition}
\begin{proof}
    For $T > 0$, we introduce the following function spaces:
    \begin{align*}
        Y_1(T) := \mathcal{C}([0,T], H^{\sigma} \cap L^{q} \cap L^{\infty}),\quad Y_2(T) := \mathcal{C}([0,T], H^{\sigma}) \text{ and } Y_0(T) := Y_1(T) \times Y_2(T),
    \end{align*}
    with the corresponding norms
    \begin{align*}
        \|u\|_{Y_1(T)} &:= \sup_{0 \leq t \leq T}\left(\bar{f_1}(t)^{-1}\|u(t,\cdot)\|_{L^q} + \bar{f_2}(t)^{-1}\|u(t,\cdot)\|_{L^\infty} + \bar{f_3}(t)^{-1} \||D|^{\sigma}u(t,\cdot)\|_{L^2}\right),\\
        \|v\|_{Y_2(T)}&:= \sup_{0 \leq t \leq T} \left(\bar{g}_1(t)^{-1}\|v(t,\cdot)\|_{L^2} + \bar{g}_2(t)^{-1} \||D|^{\sigma} v(t,\cdot)\|_{L^2}\right)
    \end{align*}
    and
    \begin{align*}
        \|(u,v)\|_{Y_0(T)} := \|u\|_{Y_1(T)} +\|v\|_{Y_2(T)}.
    \end{align*}
    In addition, with $k=0,1,2$ we define operators $\mathcal{N}_k$ on the spaces $Y_k(T)$ as follows:
    \begin{align*}
        \mathcal{N}_1[u] &:= u^{\rm lin} + \int_0^t \mathcal{K}_1(t-\tau,\cdot)*_x |v(\tau,x)|^p d\tau = u^{\rm lin} + u^{\rm non},\\
        \mathcal{N}_2[v] &:= v^{\rm lin} + \int_0^t \mathcal{K}_2(t-\tau,\cdot)*_x |u(\tau,x)|^q d\tau = v^{\rm lin} + v^{\rm non}
    \end{align*}
    and
    \begin{align*}
        \mathcal{N}_0[u,v] := (\mathcal{N}_1[u], \, \mathcal{N}_2[v]).
    \end{align*}
    Let $M$ be a suitable positive constant. With $k = 0,1,2$ we propose the following subspaces:
    \begin{align*}
        \mathbb{Y}_k(T, M) := \left\{\mathcal{G} \in Y_k(T) \text{ such that } \|\mathcal{G}\|_{Y_k(T)} \leq M \right\}.
    \end{align*}
    Our consideration can be separated into some cases as follows:
\begin{itemize}[leftmargin=*]
\item \textbf{Case 1:} Let us assume that $q < q_{\rm crit}$ and $p \geq p_{\rm crit}$, which imply $$\displaystyle\frac{pq+p+1}{2pq-p-1} = \displaystyle \max\left\{\frac{2q+1}{pq +q-2},  \frac{pq+p+1}{2pq-p-1} \right\}.$$ We choose the weights in the norms of the solution spaces as follows:
    \begin{align*}
        \bar{f_1}(t) := (1+t)^{-\frac{n}{\sigma}(1-\frac{1}{q})+1} , \quad \bar{f_2}(t):= (1+t)^{-\frac{n}{\sigma}+1}, \quad \bar{f_3}(t):= (1+t)^{-\frac{n}{4\sigma}}
    \end{align*}
    and
    \begin{align*}
        \bar{g}_1(t) := (1+t)^{-\frac{n}{4\sigma}+ \alpha}, \quad \bar{g}_2(t) := (1+t)^{-\frac{n}{4\sigma}-\frac{1}{2}+\alpha},
    \end{align*}
    where
    \begin{align*}
        \alpha := -\frac{q-1}{pq-1}\left(1-\frac{n}{2\sigma}(p-1)\right)+\frac{p-1}{pq-1}\left(1+q-\frac{n}{\sigma}(q-1)\right).
    \end{align*}
    Now we will construct a unique local (in time) small data solution $(u,v) \in \mathbb{Y}_0(T, M\varepsilon)$ to (\ref{Main.Eq.1}). Indeed, applying the fractional Gagliardo-Nirenberg inequality from Proposition \ref{FractionalG-N} with the conditions $\min\{p,q\} \geq 2$ and $\sigma < n < 2\sigma$, then carrying out some of the steps of the proof of Theorem \ref{Theorem1}, we derive
    \begin{align*}
      \|\mathcal{N}_2[v]\|_{Y_2(T)} &\leq C_1  \varepsilon \|(u_1,v_1)\|_{\mathcal{D}} + C_1 (1+T)^{1+q-\frac{n}{\sigma}(q-1)-\alpha} \|u\|_{Y_1(T)}^q\\
      &\leq C_1  \varepsilon \|(u_1,v_1)\|_{\mathcal{D}} + C_1 (1+T)^{1+q-\frac{n}{\sigma}(q-1)-\alpha} M^q\varepsilon^q,\\
      \|\mathcal{N}_1[u]\|_{Y_1(T)} &\leq C_2 \varepsilon \|(u_1,v_1)\|_{\mathcal{D}}+C_2(1+T)^{1-\frac{n}{2\sigma}(p-1)+p\alpha} \|v\|_{Y_2(T)}^p\\
      &\leq C_2 \varepsilon \|(u_1,v_1)\|_{\mathcal{D}} +C_2(1+T)^{1-\frac{n}{2\sigma}(p-1)+p\alpha} M^p\varepsilon^p,
    \end{align*}
and
    \begin{align*}
       & \|\mathcal{N}_0[u,v] -\mathcal{N}_0[\bar{u}, \bar{v}] \|_{Y_0(T)}\notag\\
       &\quad\leq C_3\|(u,v)-(\bar{u}, \bar{v})\|_{Y_0(T)}\\
       &\qquad \times \left((1+T)^{1-\frac{n}{2\sigma}(p-1)+p\alpha} \left(\|v\|_{Y_2(T)}^{p-1} + \|\bar{v}\|_{Y_2(T)}^{p-1} \right)+ (1+T)^{1+q-\frac{n}{\sigma}(q-1)-\alpha}\left(\|u\|_{Y_1(T)}^{q-1} + \|\bar{u}\|_{Y_1(T)}^{q-1}\right)\right)\notag\\
       &\quad \leq C_3 \|(u,v)-(\bar{u}, \bar{v})\|_{Y_0(T)} \left((1+T)^{1-\frac{n}{2\sigma}(p-1)+p\alpha}M^{p-1}\varepsilon^{p-1} + (1+T)^{1+q-\frac{n}{\sigma}(q-1)-\alpha}M^{q-1}\varepsilon^{q-1}\right),
    \end{align*}
    for any $(u,v),(\bar{u}, \bar{v}) \in \mathbb{Y}_0(T, M\varepsilon)$. Here, we note that
    \begin{align}
         1-\frac{n}{2\sigma}(p-1)+p\left(1+q-\frac{n}{\sigma}(q-1)\right) &> 0,\notag\\
         1+q-\frac{n}{\sigma}(q-1)-\alpha = \frac{q-1}{pq-1}\left(1-\frac{n}{2\sigma}(p-1)+p\left(1+q-\frac{n}{\sigma}(q-1)\right)\right) &> 0,\label{BT1}\\
         1-\frac{n}{2\sigma}(p-1)+p\alpha = \frac{p-1}{pq-1}\left(1-\frac{n}{2\sigma}(p-1)+p\left(1+q-\frac{n}{\sigma}(q-1)\right)\right) &> 0.\label{BT2}
    \end{align}
    Therefore, as long as $$\max\{C_1, C_2\}\varepsilon \|(u_1,v_1)\|_{\mathcal{D}} < \frac{M\varepsilon}{6},$$ and
    \begin{align}
        &\max\{C_1, C_2, C_3\} \max\Big\{(1+T)^{1-\frac{n}{2\sigma}(p-1)+p\alpha} M^p\varepsilon^{p}, (1+T)^{1+q-\frac{n}{\sigma}(q-1)-\alpha}M^q\varepsilon^{q}, \notag\\
        &\hspace{6cm} (1+T)^{1-\frac{n}{2\sigma}(p-1)+p(1+q-\frac{n}{\sigma}(q-1))} M^{pq}\varepsilon^{pq}\Big\} < \frac{M\varepsilon}{6}, \label{con1}
    \end{align}
    applying Banach's fixed point theorem we may construct a unique local solution $(u,v) \in \mathbb{Y}_0(T, M\varepsilon)$.
    Consequently, it entails the following estimates:
     \begin{align*}
      \|v\|_{Y_2(T)}
      &\leq C_1  \varepsilon \|(u_1,v_1)\|_{\mathcal{D}} + C_1 (1+T)^{1+q-\frac{n}{\sigma}(q-1)-\alpha} M^q\varepsilon^q,\\
      \|u\|_{Y_1(T)}
      &\leq C_2 \varepsilon \|(u_1,v_1)\|_{\mathcal{D}} +C_2(1+T)^{1-\frac{n}{2\sigma}(p-1)+p\alpha} \varepsilon^p  \\
      &\hspace{5cm} + C_2(1+T)^{1-\frac{n}{2\sigma}(p-1)+p(1+q-\frac{n}{\sigma}(q-1))}  M^{pq}\varepsilon^{pq}.
    \end{align*}
    Now we choose
    \begin{align*}
        T^* := \sup\left\{t \in (0, T_{\varepsilon} )\text{ such that } \mathcal{H}(t) := \|(u,v)\|_{Y_0(t)} \leq M\varepsilon\right\}.
    \end{align*}
    From  the condition (\ref{con1}), we see that $\mathcal{H}(T^*) < M\varepsilon/2$.  Due to the fact that $\mathcal{H}=\mathcal{H}(t)$ is a continuous and increasing function for any $t \in (0, T_{\varepsilon})$, we may claim that there exists a time $T^0 \in (T^*, T_{\varepsilon})$ such that $\mathcal{H}(T^0) \leq M\varepsilon$. This contradicts to the definition of $T^*$. For this reason, combining the relations (\ref{BT1}) and (\ref{BT2}) one realizes
    \begin{align*}
        C_3(1+ T^*)^{1-\frac{n}{2\sigma}(p-1)+p(1+q-\frac{n}{\sigma}(q-1))}M^{pq}\varepsilon^{pq} \geq \frac{M\varepsilon}{6},
    \end{align*}
    that is,
    \begin{align*}
        T_{\varepsilon} \geq T^* \geq c \varepsilon^{-\frac{pq-1}{1-\frac{n}{2\sigma}(p-1)+p(1+q-\frac{n}{\sigma}(q-1))}} = c \varepsilon^{-\frac{2\sigma}{\Gamma_{\rm c}(p,q)}}.
    \end{align*}
    Therefore, in this case the proof of Proposition \ref{proposition4.2}  is established.
\item \textbf{Case 2:} Next, let us consider $q \geq q_{\rm crit}$ and $p < p_{\rm crit}$. In this case, we choose the weights in the norms of the solution spaces as follows:
 $$ \bar{f_1}(t) := (1+t)^{-\frac{n}{\sigma}(1-\frac{1}{q})+1 -\alpha}, \quad \bar{f_2}(t):= (1+t)^{-\frac{n}{\sigma}+1 -\alpha}, $$
$$ \bar{f_3}(t):= (1+t)^{-\frac{n}{4\sigma} -\alpha} $$
    and
$$ \bar{g}_1(t) := (1+t)^{-\frac{n}{4\sigma}}, \quad \bar{g}_2(t) := (1+t)^{-\frac{n}{4\sigma}-\frac{1}{2}}. $$
    Then, following some steps as above we can conclude Proposition \ref{proposition4.2} in this case, too.
\item \textbf{Case 3:} Finally, let us consider the case $p < p_{\rm crit}$ and $q < q_{\rm crit}$. Then, we choose the weights in the norms of the solution spaces either as the case $q < q_{\rm crit}, \,p \geq p_{\rm crit}$ if $\alpha > 0$ or as the case $q \geq q_{\rm crit}, p < p_{\rm crit}$ if $\alpha \leq 0$.
    Then, carrying out some steps as above we may arrive at Proposition \ref{proposition4.2} in this case, too.
\end{itemize}
Summarizing, the proof of Proposition \ref{proposition4.2} is completed.
\end{proof}

\begin{remark}
\fontshape{n}
\selectfont
Linking the achieved estimates \eqref{Upper_Lifespan} and \eqref{Lower_Lifespan} in Propositions \ref{proposition4.1} and \ref{proposition4.2} one recognizes that the sharp lifespan estimates for solutions to the Cauchy problem \eqref{Main.Eq.1} in the subcritical case, i.e. the condition \eqref{Crit.Condition} occurs, are determined by the following relation:
$$ T_{\varepsilon} \sim \varepsilon^{-\frac{2\sigma}{\Gamma_{\rm c}(p,q)}}. $$
For this observation, it is really a challenging problem to verify whether or not a blow-up result still holds in the critical case, namely,
$$ \max\left\{\frac{2q+1}{pq+q-2},\frac{pq+p+1}{2pq-p-1}\right\} = \frac{n}{2\sigma}, $$
and how to catch the optimality of estimates for lifespan when the critical case occurs.
\end{remark}

\subsection{Final comments} \label{Finalcomments}
Finally, we give some comments to the following model:
\begin{equation} \label{Main.Eq.2}
\begin{cases}
u_{tt} -\Delta u -\Delta u_t = |v|^p, & x\in \R^n,\, t> 0, \\
v_{tt} -\Delta v + v_t = |u|^q, & x \in \mathbb{R}^n,\, t> 0,\\
u(0,x)= 0,\quad u_t(0,x)= u_1(x), & x\in \R^n, \\
v(0,x) = 0,\, \quad v_t(0,x) = v_1(x), & x \in \mathbb{R}^n,
\end{cases}
\end{equation}
that is, (\ref{Main.Eq.1}) with $\sigma=1$. Then, we have a visco-elastic damping term in the first equation and a friction term in the second equation. As mentioned in Remark \ref{remark2.2} the obtained estimates in Proposition \ref{Proposition2.1} are no longer valid in the case $\sigma=1$. At this point, let us refer the following estimates instead:
\begin{proposition}[see Theorem 6 in \cite{DabbiccoReissig2014}]\label{Proposition2.3}
    Let $n \geq 2$. The solution to \eqref{Problem1} with $\sigma=1$ satisfies the $(L^1 \cap L^2) - L^2$ estimates
\begin{align*}
    \|u(t,\cdot)\|_{L^{2}} &\lesssim
    \begin{cases}
        (1+t)^{-\frac{n}{4}+\frac{1}{2}} \|u_1\|_{L^1 \cap L^2} &\text{ if } n \geq 3,\\
        \log(e+t) \|u_1\|_{L^1 \cap L^2} &\text{ if } n=2,
    \end{cases} \\
    \|(\nabla u, u_t)(t,\cdot)\|_{L^2} &\lesssim (1+t)^{-\frac{n}{4}} \|u_1\|_{L^1 \cap L^2},\\
    \|\nabla^2 u(t,\cdot)\|_{L^2} &\lesssim (1+t)^{-\frac{n}{4}-\frac{1}{2}} \|u_1\|_{L^1 \cap L^2},
\end{align*}
and the $L^2 - L^2$ estimates
\begin{align*}
    \|(\nabla u, u_t)(t,\cdot)\|_{L^2} &\lesssim \|u_1\|_{L^2},\\
    \|\nabla^2 u(t,\cdot)\|_{L^2} &\lesssim (1+t)^{-\frac{1}{2}} \|u_1\|_{L^2}.
\end{align*}
\end{proposition}
By the aid of Propositions \ref{Proposition2.2} and \ref{Proposition2.3} and after repeating the proof of Theorem \ref{Theorem1}, we may conclude he following result:
\begin{proposition}
    Let $n \geq 3$. Assume that the exponents $p, q$ satisfy $2 \leq q < \infty $ if $n = 3,4$ or $2 \leq q \leq n/(n-4)$ if $n > 4$, and $2 \leq p \leq n/(n-2)$. Moreover, we also  assume the condition
    \begin{align}
        \max\left\{\frac{3q/2 +1}{pq-1}, \,\frac{pq/2 +p+1}{pq-1}\right\} < \frac{n}{2}. \label{condition-remark4.2}
    \end{align}
    Then, there exists a constant $\varepsilon_0 > 0$ such that for any small data
$ (u_1,v_1) \in \mathcal{D} $
fulfilling the assumption $ \|(u_1,v_1)\|_{\mathcal{D}} < \varepsilon_0$,
we have a uniquely determined global (in time) small data Sobolev solution
\begin{align*}
    (u,v) \in \left(\mathcal{C}([0,\infty), H^2) \cap C^1([0,\infty), L^2)\right) \times \left(\mathcal{C}([0,\infty), H^{1}) \cap C^1([0,\infty), L^2)\right)
\end{align*}
to \eqref{Main.Eq.2}.
\end{proposition}
If we replace $\sigma = 1$ in the assumption (\ref{condition1.1.1}) of Theorem \ref{Theorem1} formally, then it is easy to see that
\begin{align*}
    \frac{3q/2 + 1}{pq -1} > \frac{2q+1}{pq+q-2} \,\text{ and } \,\frac{pq/2 +p+1}{pq-1} > \frac{pq+p+1}{2pq-p-1}
\end{align*}
under the condition $2 \leq p \leq n/(n-2) \leq 3$. Therefore, we can say that the assumption (\ref{condition1.1.1}) in the case $\sigma = 1$ is really stronger than the assumption (\ref{condition-remark4.2}). From this observation linked to the fact that Theorem \ref{Theorem3} still holds in the case $\sigma=1$, there exists a gap between the following curves:
$$ \max\left\{\frac{2q+1}{pq +q-2}, \frac{pq+p+1}{2pq-p-1} \right\} -\frac{n}{2}= 0, $$
and
$$ \max\left\{\frac{3q/2 +1}{pq-1}, \,\frac{pq/2 +p+1}{pq-1}\right\} - \frac{n}{2} = 0. $$
The appearance of this gap is reasonable since it still leaves as an open problem so far to verify the critical exponent for the following Cauchy problem:
$$ \begin{cases}
u_{tt} -\Delta u -\Delta u_t = |u|^p, & x\in \R^n,\, t> 0, \\
u(0,x)= u_0(x),\quad u_t(0,x)= u_1(x), & x\in \R^n.
\end{cases} $$
We expect that once the critical exponent for the previous Cauchy problem can be determined, it gives a good chance to demonstrate the critical curve for \eqref{Main.Eq.2}.

\section*{Acknowledgments}
This research was partly supported by Vietnam Ministry of Education and Training under grant number B2023-BKA-06.


\end{document}